\documentclass[11pt,letterpaper,oneside,english]{amsart}
\usepackage[T1]{fontenc}
\usepackage[latin9]{inputenc}
\usepackage{amstext}
\usepackage{amsthm}
\usepackage{amssymb}
\PassOptionsToPackage{normalem}{ulem}
\usepackage{ulem}

\makeatletter

\pdfpageheight\paperheight
\pdfpagewidth\paperwidth

\numberwithin{equation}{section}
\numberwithin{figure}{section}
\theoremstyle{plain}
\newtheorem{thm}{\protect\theoremname}
\theoremstyle{definition}
\newtheorem{defn}[thm]{\protect\definitionname}
\theoremstyle{plain}
\newtheorem*{thm*}{\protect\theoremname}
\theoremstyle{definition}
\newtheorem*{defn*}{\protect\definitionname}
\theoremstyle{remark}
\newtheorem{rem}[thm]{\protect\remarkname}
\theoremstyle{plain}
\newtheorem{prop}[thm]{\protect\propositionname}
\theoremstyle{plain}
\newtheorem{lem}[thm]{\protect\lemmaname}
\theoremstyle{definition}
\newtheorem{example}[thm]{\protect\examplename}
\theoremstyle{plain}
\newtheorem{cor}[thm]{\protect\corollaryname}

\makeatother

\usepackage{babel}
\providecommand{\corollaryname}{Corollary}
\providecommand{\definitionname}{Definition}
\providecommand{\examplename}{Example}
\providecommand{\lemmaname}{Lemma}
\providecommand{\propositionname}{Proposition}
\providecommand{\remarkname}{Remark}
\providecommand{\theoremname}{Theorem}

\begin{document}
\title{Poisson-Orlicz norm and infinite ergodic theory}
\author{Emmanuel Roy}
\address{Laboratoire Analyse G\'eom\'etrie et Applications, Universit\'e Paris 13,
Institut Galil\'ee, 99 avenue Jean-Baptiste Cl\'ement, 93430 Villetaneuse,
France}
\email{roy@math.univ-paris13.fr}
\begin{abstract}
Urbanik's theorem for a Poisson process on an infinite measure space
$\left(X,\mathcal{A},\mu\right)$ relates integrability of stochastic
integrals to a particular Orlicz function space $L^{\Phi}\left(\mu\right)$
on which the $L^{1}$-norm of the Poisson process induces a norm (called
\emph{Poisson-Orlicz} in the sequel) that is shown to be equivalent
to the classical gauge and Orlicz norms. 

We obtain a full characterization of stochastic integrals using difference
operators that, together with a simple duality argument, allows to
derive Urbanik's theorem as well as an optimal inequality between
the Orlicz and the Poisson-Orlicz norm.

In a second part, we show that the Poisson-Orlicz norm plays a role
in infinite Ergodic Theory where it is seen as an alternative to the
$L^{1}$-norm to identify several dynamical invariants that the latter
fails to identify. We also show that, whereas the $L^{1}$-norm fully
characterizes \emph{exact} endomorphisms (Lin's theorem), Poisson-Orlicz
norm fully characterizes \emph{remotely infinite} endomorphisms.
\end{abstract}

\maketitle

\section{Introduction}

\subsection{The context: Integrability of stochastic integrals, Urbanik's theorem}

Let $\left(X,\mathcal{A},\mu\right)$ be a non-atomic \uline{infinite}
measure Borel space. To this space, we can always associate a probability
space $\left(X^{*},\mathcal{A}^{*},\mu^{*}\right)$ (Section \ref{subsec:Poisson-space})
that naturally supports a Poisson point process with intensity $\mu$.
This space can be considered as a companion space to $\left(X,\mathcal{A},\mu\right)$
that offers a probabilistic point of view to this infinite measure
object. Our starting point is the observation that some elementary
integrability properties on $\left(X,\mathcal{A},\mu\right)$ are
seen through integrability properties of the so-called \emph{stochastic
integrals} (Section \ref{subsec:Stochastic-integrals}) with respect
to the Poisson point process seen as a random element of $\left(X^{*},\mathcal{A}^{*},\mu^{*}\right)$,
it can be formulated this way: Which functions $f\in L^{0}\left(\mu\right)$
can be integrated against the Poisson point process ? This question
is well understood (see \cite{Maruyama70IDproc} for example): upon
compensating this point process, a stochastic integral $I\left(f\right)$
is well defined if and only $f$ satisfies $\int_{X}f^{2}\wedge1d\mu<+\infty$.
It happens that the set
\[
L^{\chi}\left(\mu\right):=\left\{ f\in L^{0}\left(\mu\right),\:\int_{X}f^{2}\wedge1d\mu<+\infty\right\} 
\]

is a so-called \emph{generalized Orlicz linear space} (see \cite{Rao1991},
page 400). It illustrates the following idea: some function spaces
over $\left(X,\mathcal{A},\mu\right)$ have a very natural interpretation
in terms of Poisson space. This idea culminates when looking at $L^{2}\left(\mu\right)$
(Section \ref{sec:Square-integrable-stochastic-integrals}) since
we have the equivalence
\[
I\left(f\right)\in L^{2}\left(\mu^{*}\right)\Longleftrightarrow f\in L^{2}\left(\mu\right),
\]

and, setting $I_{1}\left(f\right):=I\left(f\right)-\mathbb{E}_{\mu^{*}}\left[I\left(f\right)\right]$,
we have
\[
\left\Vert I_{1}\left(f\right)\right\Vert _{L^{2}\left(\mu^{*}\right)}=\left\Vert f\right\Vert _{L^{2}\left(\mu\right)}.
\]

This has tremendous consequences as it yields the decomposition of
$L^{2}\left(\mu^{*}\right)$ as the Fock space over $L^{2}\left(\mu\right)$
and this plays a great role in quantum physics, stochastic geometry,
Poisson-Malliavin calculus, etc... Less neglected and our main focus
in this paper is the integrable case (Section \ref{sec:Integrable-stochastic-integrals}):
Letting $\Phi$ be defined on $\mathbb{R}_{+}$ by
\[
x\mapsto\begin{cases}
x^{2} & \text{if }x\le1\\
2x-1 & \text{if }x>1
\end{cases},
\]

$\Phi$ is then a so-called \emph{Young function} of the\emph{ Orlicz
space} $L^{\Phi}\left(\mu\right)$
\[
L^{\Phi}\left(\mu\right):=\left\{ f\in L^{0}\left(\mu\right),\:\int_{X}\Phi\left(\left|f\right|\right)d\mu<\infty\right\} 
\]

that turns to be a Banach space when endowed with one of the two naturally
defined equivalent norms, the \emph{gauge} norm $N_{\Phi}$ or the
\emph{Orlicz} norm $\left\Vert \cdot\right\Vert _{\Phi}$ associated
to $\Phi$. It can then be verified that
\[
I\left(f\right)\in L^{1}\left(\mu^{*}\right)\Longleftrightarrow f\in L^{\Phi}\left(\mu\right),
\]

and we have the remarkable result of Urbanik that we rephrase in our
context:
\begin{thm}
(Urbanik, \cite{Urbanik1967}) The formula

\[
\left\Vert f\right\Vert _{*}:=\left\Vert I_{1}\left(f\right)\right\Vert _{L^{1}\left(\mu^{*}\right)}
\]

defines a norm on $L^{\Phi}\left(\mu\right)$ that is equivalent to
the gauge or Orlicz norm. In particular $\left(L^{\Phi}\left(\mu\right),\left\Vert \cdot\right\Vert _{*}\right)$
is a Banach space.
\end{thm}

The Young function defining $L^{\Phi}\left(\mu\right)$ is highly
non-unique and so are the gauge or Orlicz norms, that is, replacing
$\Phi$ by another appropriately chosen \emph{Young function} would
lead to another equivalent gauge or Orlicz norm on the same space
$L^{\Phi}\left(\mu\right)$. In contrast, $\left\Vert \cdot\right\Vert _{*}$
only depends on the measure $\mu$, we propose to call this norm the
\emph{Poisson-Orlicz} norm.

\subsection{A characterization of stochastic integrals, a proof of Urbanik's
theorem with an optimal bound.}

Our first result consists into giving a complete characterization
of stochastic integrals (Theorem \ref{prop:difference operators stoch int})
by means of \emph{difference operators} (which is a popular tool used
in the quantum physics usage of Poisson processes) that is the full
generalization of a result in \cite{LastPen11Fock} that was restricted
to the $L^{2}$-case and might have an interest on its own.

As a nearly immediate consequence, we get the ``Banach space part''
of Urbanik's theorem. The equivalence of norms being derived by very
simple duality considerations, leading to an optimal inequality: for
any $f\in L^{\Phi}\left(\mu\right)$,
\[
\left\Vert f\right\Vert _{*}\le\left\Vert f\right\Vert _{\Phi}.
\]

We observe also that from the general theory of Orlicz space (\cite{Rao1991}),
for any $f\in L^{\Phi}\left(\mu\right)$,

\[
\left\Vert f\right\Vert _{\Phi}\le2N_{\Phi}\left(f\right)
\]

that therefore yields
\[
\left\Vert f\right\Vert _{*}\le2N_{\Phi}\left(f\right),
\]

improving a result of Marcus and Rosinski (\cite{Marcus2001}) that
obtained
\[
\left\Vert f\right\Vert _{*}\le\left(2.125\right)N_{\Phi}\left(f\right).
\]

\subsection{Poisson-Orlicz norm and infinite Ergodic theory}

We then turn to Ergodic Theory, adding a measure preserving transformation
$T$ on $\left(X,\mathcal{A},\mu\right)$ (and the corresponding $T_{*}$
on $\left(X^{*},\mathcal{A}^{*},\mu^{*}\right)$) to the picture (Section
\ref{subsec:Poisson-suspension}). Our aim is to show the interest
of the Poisson-Orlicz norm within infinite Ergodic Theory allowing
to characterize some of the main dynamical invariants.

\subsubsection{$L^{1}$-Birkhoff sums convergence in infinite measure}

We recall this classical fact: if $\left(X,\mathcal{A},\mu,T\right)$
is an ergodic dynamical system then, whereas the Birkhoff sums, $\frac{1}{n}\sum_{k=1}^{n}f\circ T^{k}\left(x\right)$,
associated to any $f\in L^{1}\left(\mu\right)$ converge pointwise
to $0$ as $n$ tends to $+\infty$ for $\mu$-a.e. $x\in X$ (see
\cite{Hopf1937}), the convergence cannot hold for all $f$ in $L^{1}\left(\mu\right)$
with respect to $\left\Vert \cdot\right\Vert _{1}$. The same result
occurs when we replace the ergodicity requirement by the absence of
absolutely continuous $T$-invariant probability measure. As such
the $\left(L^{1}\left(\mu\right),\left\Vert \cdot\right\Vert _{1}\right)$-isometry
$T$ doesn't belong to the class of ``\emph{mean ergodic operators}'':
\begin{defn}
(see Section 8.4, page 136 in \cite{Eisner2015}) Let $\varphi$ be
a continuous linear operator defined on a Hausdorff topological vector
space $H$. $\varphi$ is said to be \emph{mean ergodic} if, for any
$h\in H$, $\frac{1}{n}\sum_{k=1}^{n}\varphi^{k}h$ converges as $n$
tends to $+\infty$.
\end{defn}

In contrast, replacing $\left(L^{1}\left(\mu\right),\left\Vert \cdot\right\Vert _{1}\right)$
by $\left(L^{p}\left(\mu\right),\left\Vert \cdot\right\Vert _{p}\right)$,
$1<p<\infty$, $T$ becomes a mean ergodic operator and one has:
\begin{thm*}
The dynamical system $\left(X,\mathcal{A},\mu,T\right)$ has no absolutely
continuous $T$-invariant probability measure if and only if, for
every $f\in L^{p}\left(\mu\right)$:
\[
\frac{1}{n}\sum_{k=1}^{n}f\circ T^{k}\to_{\left\Vert \cdot\right\Vert _{p}}0,
\]

as $n$ tends to $+\infty$.
\end{thm*}
The fact that $T$ is a mean ergodic operator on $\left(L^{p}\left(\mu\right),\left\Vert \cdot\right\Vert _{p}\right)$,
$1<p<\infty$, follows, for example, by the fact that they are, unlike
$\left(L^{1}\left(\mu\right),\left\Vert \cdot\right\Vert _{1}\right)$,
reflexive spaces, by a result due to Lorch (1939).

Let us illustrate another instance where $\left\Vert \cdot\right\Vert _{1}$
can't help in identifying another invariant:
\begin{defn*}
$\left(X,\mathcal{A},\mu,T\right)$ is of \emph{zero type} if for
every sets $A$ and $B$ in $\mathcal{A}_{f}$:
\[
\mu\left(A\cap T^{-n}B\right)\to0
\]
\end{defn*}
as $n$ tends to $+\infty$, 

We recall this classical result:
\begin{thm*}
(Blum-Hanson, \cite{Blum1960}) Let $\left(\Omega,\mathcal{F},\mathbb{P},S\right)$
be a \uline{probability} measure-preserving dynamical system and
let $1\le p<\infty$. Then $\left(\Omega,\mathcal{F},\mathbb{P},S\right)$
is mixing if and only if, for every strictly increasing sequence $\left\{ n_{k}\right\} _{k\in\mathbb{N}}$
of integers and every $F\in L^{p}\left(\mathbb{P}\right)$:
\[
\frac{1}{n}\sum_{k=1}^{n}F\circ S^{n_{k}}\to_{\left\Vert \cdot\right\Vert _{p}}\mathbb{E}_{\mathbb{P}}\left[F\right]
\]

as $n$ tends to $+\infty$.
\end{thm*}
As shown in \cite{KrenSuch69MixInf}, replacing $\left(\Omega,\mathcal{F},\mathbb{P}\right)$
by an infinite-measure space, ``mixing'' by ``zero type'' and
the limit by $0$, then the theorem holds for $1<p<\infty$ but fails
of course again for $p=1$.

\subsubsection{$\left\Vert \cdot\right\Vert _{*}$ as an intrinsic alternative to
$\left\Vert \cdot\right\Vert _{1}$}

Replacing $\left\Vert \cdot\right\Vert _{1}$ by $\left\Vert \cdot\right\Vert _{*}$,
we are then able to prove (Theorem \ref{thm:ergodicL12}):
\begin{thm*}
The dynamical system $\left(X,\mathcal{A},\mu,T\right)$ has no absolutely
continuous $T$-invariant probability measure if and only if, for
every $f\in L^{1}\left(\mu\right)$ (or $f\in L^{\Phi}\left(\mu\right)$):
\[
\frac{1}{n}\sum_{k=1}^{n}f\circ T^{k}\to_{\left\Vert \cdot\right\Vert _{*}}0
\]

as $n$ tends to $+\infty$.
\end{thm*}
And (Theorem \ref{thm:Blum-Hanson_type}):
\begin{thm*}
The dynamical system $\left(X,\mathcal{A},\mu,T\right)$ is of zero
type if and only if for every strictly increasing sequence $\left\{ n_{k}\right\} _{k\in\mathbb{N}}$
of integers and every $f\in L^{1}\left(\mu\right)$ (or $f\in L^{\Phi}\left(\mu\right)$):
\[
\frac{1}{n}\sum_{k=1}^{n}f\circ T^{n_{k}}\to_{\left\Vert \cdot\right\Vert _{*}}0
\]

as $n$ tends to $+\infty$.
\end{thm*}
Thus $\left\Vert \cdot\right\Vert _{*}$ helps to get the ``correct
rate'' for the above convergence to characterize these two dynamical
invariants.

Then, mimicking the construction of the transfer operator $\widehat{T}$
on $L^{1}\left(\mu\right)$ obtained as the predual operator of $T$
acting on $L^{\infty}\left(\mu\right)$, we get an operator $\widehat{T_{\mathcal{P}}}$
acting on $L^{\Phi}\left(\mu\right)$ (Definition \ref{def:DualT_P})
and we show that not only does it preserve $L^{1}\left(\mu\right)$,
but it coincides with $\widehat{T}$ on it (Proposition \ref{prop:CoincidenceDual}).

We now need to recall these two other definitions (see \cite{KrenSuch69MixInf})
to explain our results involving the behavior of $\widehat{T}$ relatively
to $\left\Vert \cdot\right\Vert _{1}$ and $\left\Vert \cdot\right\Vert _{*}$.
\begin{defn*}
$\left(X,\mathcal{A},\mu,T\right)$ is said to be
\begin{itemize}
\item \emph{remotely infinite} if $\cap_{n\ge0}T^{-n}\mathcal{A}$ contains
only sets of zero or infinite $\mu$-measure.
\item \emph{exact} if $\cap_{n\ge0}T^{-n}\mathcal{A}=\left\{ \emptyset,X\right\} $
mod. $\mu$.
\end{itemize}
\end{defn*}
In particular, as $\mu\left(X\right)=+\infty$, we have the chain
of implications:
\[
\text{exact}\Longrightarrow\text{remotely infinite}\Longrightarrow\text{zero type}\Longrightarrow\text{absence of a.c. \ensuremath{T}-invariant measure}.
\]

Set $L_{0}^{1}\left(\mu\right):=\left\{ f\in L^{1}\left(\mu\right),\,\int_{X}fd\mu=0\right\} $.
Lin showed the following characterization of exactness:
\begin{thm*}
(\cite{Lin1971}) The dynamical system $\left(X,\mathcal{A},\mu,T\right)$
is exact if and only if for every $f\in L_{0}^{1}\left(\mu\right)$:
\[
\left\Vert \widehat{T^{n}}f\right\Vert _{1}\to0
\]

as $n$ tends to $+\infty$.
\end{thm*}
Then, adapting the rate with $\left\Vert \cdot\right\Vert _{*}$ we
get (Theorem \ref{thm:remotelyDualOp}):
\begin{thm*}
The dynamical system $\left(X,\mathcal{A},\mu,T\right)$ is remotely
infinite if and only if for every $f\in L^{\Phi}\left(\mu\right)$:
\[
\left\Vert \widehat{T_{\mathcal{P}}^{n}}f\right\Vert _{*}\to0
\]

or equivalently for every $f\in L_{0}^{1}\left(\mu\right)$ (or $f\in L^{1}\left(\mu\right)$)
\[
\left\Vert \widehat{T^{n}}f\right\Vert _{*}\to0
\]

as $n$ tends to $+\infty$.
\end{thm*}

\section{\label{sec:Background-on-Poisson}Background on the Poisson space
and stochastic integrals}

\subsection{\label{subsec:Poisson-space}Poisson space}

Throughout this paper $\left(X,\mathcal{A},\mu\right)$ is a non-atomic
infinite Borel space, that is, measurably isomorphic to the real line
endowed with Borel sets (for the usual topology) and Lebesgue measure.

We recall a possible definition of the \emph{Poisson space} (also
called \emph{Poisson measure} or \emph{Poisson point process}) $\left(X^{*},\mathcal{A}^{*},\mu^{*}\right)$
over $\left(X,\mathcal{A},\mu\right)$:
\begin{itemize}
\item $X^{*}$ is the collection of measures of the form $\nu:=\sum_{i\in I}\delta_{x_{i}}$,
$x_{i}\in X$, $I$ countable.
\item $\mathcal{A}^{*}:=\sigma\left\{ N\left(A\right),\,A\in\mathcal{A}\right\} $
where $N\left(A\right)$ is the map $\omega\in X^{*}\mapsto\omega\left(A\right)$.
\item $\mu^{*}$ is the only probability measure such that, for any $k\in\mathbb{N}$,
any collection $A_{1},\dots,A_{k}$ of pairwise disjoint sets in $\mathcal{A}_{f}$,
the random variables $N\left(A_{1}\right),\dots,N\left(A_{k}\right)$
are independent and Poisson distributed with parameter $\mu\left(A_{1}\right),\dots,\mu\left(A_{k}\right)$
respectively.
\end{itemize}
It can be checked out that, once $\mathcal{A}^{*}$ is completed with
respect to $\mu^{*}$, up to a negligible set, the Poisson space becomes
a Lebesgue probability space (a standard Borel space with the completion
of the Borel $\sigma$-algebra). In other words, it is a ``nice''
probability space.

\subsection{\label{subsec:Poisson-suspension}Poisson suspension}

If $T$ is an endomorphism of $\left(X,\mathcal{A},\mu\right)$ then
the map $T_{*}$ defined by
\[
\omega\in X^{*}\mapsto\omega\circ T^{-1},
\]

is an endomorphism of $\left(X^{*},\mathcal{A}^{*},\mu^{*}\right)$.

The dynamical system $\left(X^{*},\mathcal{A}^{*},\mu^{*},T_{*}\right)$
is the \emph{Poisson suspension} over the \emph{base} $\left(X,\mathcal{A},\mu,T\right)$
and we can present the most basic ergodic results, obtained by Marchat
(\cite{March78Poiss}):
\begin{thm}
\label{thm:Ergodic basics}Let $T$ be an automorphism of $\left(X,\mathcal{A},\mu\right)$.
\begin{itemize}
\item $\left(X^{*},\mathcal{A}^{*},\mu^{*},T_{*}\right)$ is ergodic (and
weakly mixing) if and only if $\left(X,\mathcal{A},\mu,T\right)$
doesn't possess any $T$-invariant set of positive and finite $\mu$-measure,
or equivalently, if there is no absolutely continuous $T$-invariant
probability measure.
\item $\left(X^{*},\mathcal{A}^{*},\mu^{*},T_{*}\right)$ is mixing if and
only if $\left(X,\mathcal{A},\mu,T\right)$ is of zero type.
\end{itemize}
\end{thm}

In particular, as we suppose $\mu$ infinite, $\left(X^{*},\mathcal{A}^{*},\mu^{*},T_{*}\right)$
is ergodic as soon as $\left(X,\mathcal{A},\mu,T\right)$ is.
\begin{rem}
Most references on Poisson suspensions deal with automorphisms (see
for example \cite{Roy07Infinite}) however Zweim\"uller (\cite{Zweimueller2008})
dealt with endomorphisms and proved that the natural extension of
the Poisson suspension over some system is the Poisson suspension
of the natural extension of this same system. The above theorem is
thus valid for endomorphisms as well.
\end{rem}

\subsection{\label{subsec:Stochastic-integrals}Stochastic integrals}

In the following, depending on the context, when dealing with the
integral of a function $g$ with respect to a measure $\rho$ on the
space $X$, we'll use the following equivalent notation:
\begin{itemize}
\item $\int_{X}g\left(x\right)\rho\left(dx\right)$
\item $\int_{X}gd\rho$
\item $\rho\left(g\right)$
\end{itemize}
However, we'll stick to the probabilistic notation $\mathbb{E}_{\mu^{*}}\left[F\right]$
for the expectation of a random variable $F$ on the probability space
$\left(X^{*},\mathcal{A}^{*},\mu^{*}\right)$.

We detail the notion of \emph{stochastic integral} in the Poisson
setting whose complete construction can be found in (\cite{Maruyama70IDproc}).

The starting point is the following formula readily available from
the definition: for any $A\in\mathcal{A}_{f}$
\[
\mathbb{E}_{\mu^{*}}\left[N\left(A\right)\right]=\mu\left(A\right),
\]

This allows to safely define such quantity as $N\left(f\right)$ on
$X^{*}$ for an $L^{1}\left(\mu\right)$-function $f$ on $\left(X,\mathcal{A}\right)$:

\[
N\left(f\right):\;\omega\mapsto\omega\left(f\right),
\]

at least for $\mu^{*}$-a.e. $\omega\in X^{*}$.

Indeed, if $f\in L^{1}\left(\mu\right)$, using standard monotone
approximation arguments, $N\left(\left|f\right|\right)$ is easily
proven to be in $L^{1}\left(\mu^{*}\right)$ and satisfies:
\[
\mathbb{E}_{\mu^{*}}\left[N\left(\left|f\right|\right)\right]=\mu\left(\left|f\right|\right).
\]

Therefore, for any $f\in L^{1}\left(\mu\right)$, $N\left(\left|f\right|\right)$
is finite $\mu^{*}$-a.e. and thus $f$ is $\omega$-integrable for
$\mu^{*}$-a.e. $\omega\in X^{*}$ and we have:
\begin{equation}
\mathbb{E}_{\mu^{*}}\left[N\left(f\right)\right]=\mu\left(f\right).\label{eq:Integralebis-1}
\end{equation}

Observe that replacing $f$ by some function $\widetilde{f}$ such
that $f=\widetilde{f}$ $\mu$-a.e. yields two random variables $N\left(f\right)$
and $N\left(\widetilde{f}\right)$ that are equal $\mu^{*}$-a.e..

We define, for any $f\in L^{1}\left(\mu\right)$
\begin{equation}
I_{1}\left(f\right):=N\left(f\right)-\mu\left(f\right).\label{eq:I_1}
\end{equation}

It is a centered and integrable random variable.

It is however possible to get further. Set $\chi\left(x\right)=x^{2}\wedge1$,
$x\ge0$ and consider the \emph{generalized Orlicz space} (see \cite{Rao1991},
page 400):
\[
L^{\chi}\left(\mu\right):=\left\{ f\in L^{0}\left(\mu\right),\:\int_{X}\chi\left(\left|f\right|\right)d\mu<+\infty\right\} .
\]

Then the \emph{stochastic integral} $I\left(f\right)$, for $f\in L^{\chi}\left(\mu\right)$,
is defined as

\begin{align*}
I\left(f\right) & :=\lim_{\epsilon\to0}N\left(f1_{\left|f\right|>\epsilon}\right)-\mu\left(f1_{\left|f\right|\le1}1_{\left|f\right|>\epsilon}\right)\\
 & =\lim_{\epsilon\to0}I_{1}\left(f1_{\epsilon<\left|f\right|\le\frac{1}{\epsilon}}\right)+\mu\left(f1_{1<\left|f\right|\le\frac{1}{\epsilon}}\right),
\end{align*}

the limit taking place in $\mu^{*}$-measure.
\begin{rem}
\label{rem:Almostlinear}Pay attention to the fact that $I$ is not
linear, nevertheless we have:
\[
I\left(\alpha f+g\right)=\alpha I\left(f\right)+I\left(g\right)+c
\]

for some constant $c$ depending on $\alpha$, $f$ and $g$.

In particular, we have
\[
I\left(f\right)=I\left(\mathfrak{Re}f\right)+iI\left(\mathfrak{Im}f\right)+d
\]

for some constant $d$ depending on $f$.

However, the effect of an endomorphism doesn't involve the addition
of a constant, thanks to the fact it preserves the measure:

\begin{align*}
I\left(f\right)\circ T_{*} & =\lim_{\epsilon\to0}N\left(f1_{\left|f\right|>\epsilon}\right)\circ T_{*}-\mu\left(f1_{\left|f\right|\le1}1_{\left|f\right|>\epsilon}\right)\\
 & =\lim_{\epsilon\to0}N\left(f\circ T1_{\left|f\circ T\right|>\epsilon}\right)-\mu\left(f1_{\left|f\right|\le1}1_{\left|f\right|>\epsilon}\right)\\
 & =\lim_{\epsilon\to0}N\left(f\circ T1_{\left|f\circ T\right|>\epsilon}\right)-\mu\left(f\circ T1_{\left|f\circ T\right|\le1}1_{\left|f\circ T\right|>\epsilon}\right)\\
 & =I\left(f\circ T\right)
\end{align*}

Lastly, $I\left(f\right)$ is constant if and only if $f$ vanishes
$\mu$-a.e..
\end{rem}

\begin{defn}
A \emph{stochastic integral} defined on $\left(X^{*},\mathcal{A}^{*},\mu^{*}\right)$
is a random variable of the form
\[
I\left(f\right)+c,
\]

with $f$ in $L^{\chi}\left(\mu\right)$ and $c\in\mathbb{\mathbb{C}}$.
\end{defn}

In particular, when $f\in L^{1}\left(\mu\right)$, $I_{1}\left(f\right)$
is a stochastic integral since $I_{1}\left(f\right)=I\left(f\right)-\mu\left(f1_{\left|f\right|>1}\right)$.
\begin{rem}
\label{rem:ID Random variable}It is worth mentioning that a stochastic
integral $I\left(f\right)+c$ is always an \emph{infinitely divisible}
random variable (on $\mathbb{C}\simeq\mathbb{R}^{2}$) variable without
Gaussian part (see \cite{Maruyama70IDproc} together with Theorem
8.1, page 37 in \cite{Sato99LevPro}) whose L\'evy measure is the image
measure of $\mu$ by $f$, restricted to $\mathbb{C}\setminus\left\{ 0\right\} $.
Up to the addition of a constant, the L\'evy measure completely characterizes
the distribution of a stochastic integral.
\end{rem}

\subsubsection{Campbell measure and difference operators}

To derive useful properties of stochastic integrals, it is convenient
to consider a representation of them in the product space $\left(X\times X^{*},\mathcal{A}\otimes\mathcal{A}^{*},\mu\otimes\mu^{*}\right)$
as it is done in \cite{LastPen11Fock} for the $L^{2}$-case. We need
however to go beyond $L^{2}$ and we have been unable to find all
the results we want in the existing literature, we thereafter give
proofs of what is, as far as we know, new material.
\begin{defn}
(see for example \cite{Daley2008}, page 269) The \emph{Campbell measure}
$Q$ of $\left(X^{*},\mathcal{A}^{*},\mu^{*}\right)$ is defined on
$\left(X\times X^{*},\mathcal{A}\otimes\mathcal{A}^{*}\right)$, for
any positive measurable $\varphi$ by
\[
\int_{X\times X^{*}}\varphi\left(x,\omega\right)Q\left(d\left(x,\omega\right)\right):=\int_{X^{*}}\left(\int_{X}\varphi\left(x,\omega\right)\omega\left(dx\right)\right)\mu^{*}\left(d\omega\right).
\]
\end{defn}

The important formula is the following:
\begin{thm}
(Mecke Formula, \cite{Meck67Form}) For any positive measurable function
$\varphi$ on $\left(X\times X^{*},\mathcal{A}\otimes\mathcal{A}^{*}\right)$:
\[
\int_{X\times X^{*}}\varphi\left(x,\omega\right)Q\left(d\left(x,\omega\right)\right)=\int_{X\times X^{*}}\varphi\left(x,\omega+\delta_{x}\right)\mu\otimes\mu^{*}\left(d\left(x,\omega\right)\right).
\]
\end{thm}

This formula simply says that $Q$ is the image measure of $\mu\otimes\mu^{*}$
under the mapping:
\[
\left(x,\omega\right)\mapsto\left(x,\omega+\delta_{x}\right).
\]

\begin{rem}
\label{rem:abscont}It is worth noting that, when $\mu$ is infinite,
the projection of $Q$ along the second coordinate is equivalent to
$\mu^{*}$ (of course, the same applies to $\mu\otimes\mu^{*}$).
To see that, take $f>0$ in $L^{1}\left(\mu\right)$, then, as $\mu$
is an infinite measure, for $\mu^{*}$-almost all $\omega\in X^{*}$,
$\omega\neq0$ and thus $\int_{X}f\left(x\right)\omega\left(dx\right)>0$
for $\mu^{*}$-almost all $\omega\in X^{*}$. We now consider the
measure $\widetilde{Q}$ on $\left(X\times X^{*},\mathcal{A}\otimes\mathcal{A}^{*}\right)$
defined by $\frac{d\widetilde{Q}}{dQ}\left(x,\omega\right):=\frac{f\left(x\right)}{\int_{X}f\left(s\right)\omega\left(ds\right)}$
so that $\widetilde{Q}\sim Q$. Now consider its projection $\overline{Q}$
on $X^{*}$:

\begin{align*}
\overline{Q}\left(\mathfrak{A}\right) & =\int_{X\times X^{*}}1_{\mathfrak{A}}\left(\omega\right)\frac{f\left(x\right)}{\int_{X}f\left(s\right)\omega\left(ds\right)}Q\left(d\left(x,\omega\right)\right)\\
 & =\int_{X^{*}}1_{\mathfrak{A}}\left(\omega\right)\frac{1}{\int_{X}f\left(s\right)\omega\left(ds\right)}\left(\int_{X}f\left(t\right)\omega\left(dt\right)\right)\mu^{*}\left(d\omega\right)\\
 & =\mu^{*}\left(\mathfrak{A}\right).
\end{align*}

Thus $\overline{Q}=\mu^{*}$.
\end{rem}

For any measurable $F$ on $\left(X,\mathcal{A}^{*}\right)$, define
$F_{\delta}$ on $\left(X\times X^{*},\mathcal{A}\otimes\mathcal{A}^{*}\right)$
by
\[
\left(x,\omega\right)\mapsto F\left(\omega+\delta_{x}\right)
\]

\begin{defn}
(See \cite{LastPen11Fock}) The \emph{difference operator} $D$ associates
a measurable $F$ on $\left(X^{*},\mathcal{A}^{*}\right)$ to a measurable
$DF$ on $\left(X\times X^{*},\mathcal{A}\otimes\mathcal{A}^{*}\right)$
by
\[
\left(x,\omega\right)\mapsto\left(D_{x}F\right)\left(\omega\right):=F_{\delta}\left(x,\omega\right)-F\left(\omega\right).
\]
\end{defn}

It has to be noted that, taking $F=\widetilde{F}$ $\mu^{*}$-a.e.
implies that $DF=D\widetilde{F}$ $\mu\otimes\mu^{*}$-a.e., indeed,
consider
\[
\left(D_{x}F\right)\left(\omega\right)-\left(D_{x}\widetilde{F}\right)\left(\omega\right)=\left(F-\widetilde{F}\right)_{\delta}\left(x,\omega\right)-\left(F-\widetilde{F}\right)\left(\omega\right)
\]

But $F=\widetilde{F}$ $\mu\otimes\mu^{*}$-a.e. and, thanks to Remark
\ref{rem:abscont}, $\left(F-\widetilde{F}\right)_{\delta}$ is distributed,
under $\widetilde{Q}\sim\mu\otimes\mu^{*}$, as $F-\widetilde{F}$
under $\mu^{*}$ and therefore vanishes $\mu\otimes\mu^{*}$-a.e. 

This proves $DF=D\widetilde{F}$ $\mu\otimes\mu^{*}$-a.e.. and, in
particular $DF$ can be safely defined on $F$ that are only defined
$\mu^{*}$-a.e.

Those operators are well suited for stochastic integrals as we get:
\begin{prop}
\label{prop:diffInt}Let $f$ be in $L^{\chi}\left(\mu\right)$. For
$\mu\otimes\mu^{*}$-a.e. $\left(x,\omega\right)\in X\times X^{*}$:
\[
D_{x}I\left(f\right)\left(\omega\right)=f\left(x\right)
\]
\end{prop}

\begin{proof}
Set $f_{n}:=f1_{\left|f\right|>\frac{1}{n}}$, $n\ge1$, we have $\mu\left\{ \left|f\right|>\frac{1}{n}\right\} <\infty$,
hence, for $\mu^{*}$-a.e. $\omega$, $\omega\left\{ \left|f\right|>\frac{1}{n}\right\} <\infty$
and we can safely write, for $\mu\otimes\mu^{*}$-a.e. $\left(\omega,x\right)$:
\begin{align*}
\left(I\left(f_{n}\right)\right)_{\delta}\left(\omega,x\right) & =\int_{\left|f\right|>\frac{1}{n}}fd\left(\omega+\delta_{x}\right)-\int_{\left|f\right|>\frac{1}{n}}f1_{\left|f\right|\le1}d\mu\\
 & =\int_{\left|f\right|>\frac{1}{n}}fd\left(\omega\right)+f\left(x\right)1_{\left|f\right|>\frac{1}{n}}\left(x\right)-\int_{\left|f\right|>\frac{1}{n}}f1_{\left|f\right|\le1}d\mu\\
 & =I\left(f_{n}\right)\left(\omega\right)+f_{n}\left(x\right)
\end{align*}

By definition $I\left(f_{n}\right)$ tends to $I\left(f\right)$ in
$\mu^{*}$-measure. In particular, there exists an increasing sequence
$\left\{ n_{k}\right\} _{k\in\mathbb{N}}$ such that $I\left(f_{n_{k}}\right)$
tends to $I\left(f\right)$ $\mu^{*}$-a.e. as $k$ tends to $+\infty$.

But this implies, thanks to Remark \ref{rem:abscont} once again,
that $\left(I\left(f_{n_{k}}\right)\right)_{\delta}$ tends to $\left(I\left(f\right)\right)_{\delta}$
$\mu\otimes\mu^{*}$-a.e. as $k$ tends to $+\infty$.

Therefore
\[
\left(I\left(f\right)\right)_{\delta}\left(\omega,x\right)=I\left(f\right)\left(\omega\right)+f\left(x\right),
\]

hence the result.

Our next goal is to show the converse, to achieve this we need the
following intermediate results.
\end{proof}
\begin{prop}
\label{prop:LastPen}(See \cite{LastPen11Fock}) If $F\in L^{2}\left(\mu^{*}\right)$
then $x\mapsto\mathbb{E}_{\mu^{*}}\left[D_{x}F\right]$ belongs to
$L^{2}\left(\mu\right)$ and, if there exists a measurable $f$ on
$\left(X,\mathcal{A}\right)$ such that $D_{x}F\left(\omega\right)=f\left(x\right)$
for $\mu\otimes\mu^{*}$-a.e. $\left(x,\omega\right)\in X\times X^{*}$
then there exists $c\in\mathbb{C}$ such that
\[
F=I\left(f\right)+c.
\]
\end{prop}

And:
\begin{lem}
\label{lem:DF=00003D0}Let $F$ be a measurable on $\left(X^{*},\mathcal{A}^{*},\right)$
such that, $\mu\otimes\mu^{*}$-a.e.,
\[
DF=0,
\]

then $F$ is constant $\mu^{*}$-a.e..
\end{lem}

\begin{proof}
For $\mu\otimes\mu^{*}$-almost all $\left(x,\omega\right)\in X\times X^{*}$
we have
\[
F_{\delta}\left(x,\omega\right)=F\left(\omega\right),
\]

thus, for every Borel set $A$, we get
\[
1_{A}\circ F_{\delta}\left(x,\omega\right)=1_{A}\circ F\left(\omega\right),
\]

thus, $\mu\otimes\mu^{*}$-a.e:
\[
D\left(1_{A}\circ F\right)=0.
\]

But $1_{A}\circ F\in L^{2}\left(\mu^{*}\right)$, therefore it is
constant by Proposition \ref{prop:LastPen}. Since this is true for
all Borel sets, this gives the result.
\end{proof}
We can now formulate our characterization of stochastic integrals:
\begin{thm}
\label{prop:difference operators stoch int} A measurable function
$F$ on $\left(X^{*},\mathcal{A}^{*}\right)$ is a stochastic integral
if and only if there exists a measurable function $f$ on $\left(X,\mathcal{A}\right)$
such that for $\mu\otimes\mu^{*}$-a.e. $\left(x,\omega\right)\in X\times X^{*}$,
\[
\left(D_{x}F\right)\left(\omega\right)=f\left(x\right),
\]

in particular there exists $c\in\mathbb{C}$ such that $F=I\left(f\right)+c$.
\end{thm}

\begin{proof}
Only one direction has to be proved, thanks to Proposition \ref{prop:diffInt}.
Assume $F$ is a measurable, real valued function $F$ on $\left(X^{*},\mathcal{A}^{*}\right)$
and there exists a measurable, real valued function $f$ on $\left(X,\mathcal{A}\right)$
such that for $\mu\otimes\mu^{*}$-a.e. $\left(x,\omega\right)\in X\times X^{*}$,
\[
\left(D_{x}F\right)\left(\omega\right)=f\left(x\right).
\]

For any $t\in\mathbb{R}$ and for $\mu\otimes\mu^{*}$-a.e. $\left(x,\omega\right)\in X\times X^{*}$
we have
\begin{align*}
\exp itF_{\delta}\left(x,\omega\right)-\exp itF\left(\omega\right) & =\exp itf\left(x\right)\exp itF\left(\omega\right)-\exp itF\left(\omega\right)\\
 & =\exp itF\left(\omega\right)\left(\exp itf\left(x\right)-1\right).
\end{align*}

Therefore
\[
\mathbb{E}_{\mu^{*}}\left[D_{x}\left(\exp itF\right)\right]=\mathbb{E}_{\mu^{*}}\left[\exp itF\right]\left(\exp itf\left(x\right)-1\right).
\]

Now from Proposition \ref{prop:LastPen}, since $\exp itF\in L^{2}\left(\mu^{*}\right)$,
then $x\mapsto\mathbb{E}_{\mu^{*}}\left[D_{x}\left(\exp itF\right)\right]\in L^{2}\left(\mu\right)$
and therefore $\exp itf-1\in L^{2}\left(\mu\right)$ and:
\[
\int_{X}\left|\exp itf\left(x\right)-1\right|^{2}\mu\left(dx\right)=4\int_{X}\sin^{2}\left(\frac{1}{2}tf\left(x\right)\right)\mu\left(dx\right)<+\infty.
\]

We can find $t_{1}$ and $t_{2}$ such that there exists $\epsilon>0$,
such that, for any $\left|y\right|>1$, $\sin^{2}\left(\frac{1}{2}t_{1}y\right)+\sin^{2}\left(\frac{1}{2}t_{2}y\right)>\epsilon$.
In particular,
\[
\int_{X}\epsilon1_{\left|f\right|>1}\mu\left(dx\right)\le\int_{X}\sin^{2}\left(\frac{1}{2}t_{1}f\left(x\right)\right)+\sin^{2}\left(\frac{1}{2}t_{2}f\left(x\right)\right)\mu\left(dx\right)<\infty
\]

that is
\[
\mu\left(\left|f\right|>1\right)<\infty
\]

and, as $\sin^{2}\left(1\right)z^{2}\le\sin^{2}\left(z\right)$ on
$\left[-1,1\right]$,
\[
\int_{X}f^{2}1_{\left|f\right|\le1}d\mu\le\int_{X}\frac{4}{t_{1}^{2}\sin^{2}\left(1\right)}\sin^{2}\left(\frac{1}{2}t_{1}f\left(x\right)\right)\mu\left(dx\right)<\infty.
\]

It follows that $f\in L^{\chi}\left(\mu\right)$ and thus the stochastic
integral $I\left(f\right)$ is well defined and satisfies, for $\mu\otimes\mu^{*}$-a.e.
$\left(x,\omega\right)\in X\times X^{*}$ $D_{x}\left(I\left(f\right)\right)\left(\omega\right)=f\left(x\right)$.

Therefore, $\mu\otimes\mu^{*}$-a.e. $D\left(F-I\left(f\right)\right)=0$
and then $F=I\left(f\right)+c$ for some $c\in\mathbb{R}$, thanks
to Lemma \ref{lem:DF=00003D0}.

We get the complex case by considering real an imaginary parts.
\end{proof}

\subsection{The set of stochastic integrals}

Denote by $\mathcal{I}_{\mu^{*}}\subset L^{0}\left(\mu^{*}\right)$
the set of stochastic integrals. As a first consequence of Theorem
\ref{prop:difference operators stoch int}, we get:
\begin{prop}
\label{prop:closnessofI} $\mathcal{I}_{\mu^{*}}$ is a closed subspace
of $L^{0}\left(\mu^{*}\right)$ with respect to convergence in measure.
\end{prop}

\begin{proof}
The fact that it is a linear space follows immediately from Remark
\ref{rem:Almostlinear}.

Let $\left\{ F_{n}\right\} _{n\in\mathbb{N}}:=\left\{ I\left(f_{n}\right)+c_{n}\right\} _{n\in\mathbb{N}}$
be a sequence of stochastic integrals converging to some random variable
$F$ with respect to $\mu^{*}$. There exists a subsequence $\left\{ F_{n_{k}}\right\} _{k\in\mathbb{N}}$
converging to $F$ $\mu^{*}$-a.e.

For the same reason already explained earlier $\left\{ F_{n_{k}\delta}\right\} _{k\in\mathbb{N}}$
converges $\widetilde{Q}$-a.e. and thus $\mu\otimes\mu^{*}$-a.e.
to $F_{\delta}$. Now from Proposition \ref{prop:diffInt}, for $\mu\otimes\mu^{*}$-a.e.
$\left(x,\omega\right)\in X\times X^{*}$:

\[
f_{n_{k}}\left(x\right)=D_{x}F_{n_{k}}\left(\omega\right)=F_{n_{k}\delta}\left(x,\omega\right)-F_{n_{k}}\left(\omega\right).
\]

Therefore $\left\{ f_{n_{k}}\right\} _{k\in\mathbb{N}}$ converges
$\mu$-a.e. to some measurable $f$ satisfying
\[
f\left(x\right)=F_{\delta}\left(x,\omega\right)-F\left(\omega\right)=D_{x}F\left(\omega\right).
\]

Thanks to Proposition \ref{prop:difference operators stoch int},
$F=I\left(f\right)+c$ for some $c\in\mathbb{C}$ as expected.
\end{proof}

\section{\label{sec:Square-integrable-stochastic-integrals}Square-integrable
stochastic integrals, Fock space structure}

In this section, we recall the classical $L^{2}$-case that relies
almost entirely on the Hilbertian structure (the details can be found
in \cite{Ner1996CatInfGroups}, Chapter 10, Section 4 and also \cite{LastPen11Fock}).

The starting point is the following isometry identity: for any $A$
and $B$ in $\mathcal{A}_{f}$,

\[
\mathbb{E}_{\mu^{*}}\left[\left(N\left(A\right)-\mu\left(A\right)\right)\left(N\left(B\right)-\mu\left(B\right)\right)\right]=\mu\left(A\cap B\right).
\]

And we get, for any functions $f$ and $g$ in $L^{1}\left(\mu\right)\cap L^{2}\left(\mu\right)$:
\[
\left\langle I_{1}\left(f\right),I_{1}\left(g\right)\right\rangle _{L^{2}\left(\mu^{*}\right)}=\left\langle f,g\right\rangle _{L^{2}\left(\mu\right)},
\]

where $I_{1}$ was defined by (\ref{eq:I_1}).

Moreover, from Remark \ref{rem:ID Random variable} together with
the square integrability criterion found in \cite{Sato99LevPro},
page 163, a stochastic integral $I\left(h\right)$ is square integrable
if and only if $h\in L^{2}\left(\mu\right)$.

Setting $\mathcal{I}_{2}:=\mathcal{I}\cap L_{0}^{2}\left(\mu^{*}\right)$
this leads immediately to:
\begin{thm}
\label{thm:L^2 isomorphism}$I_{1}$ extends to an isometric isomorphism
between the Hilbert spaces $\left(L^{2}\left(\mu\right),\left\Vert \cdot\right\Vert _{2}\right)$
and $\left(\mathcal{I}_{2},\left\Vert \cdot\right\Vert _{L^{2}\left(\mu^{*}\right)}\right)$.
In particular, for any $f$ and $g$ in $L^{2}\left(\mu\right)$:
\[
\left\langle I_{1}\left(f\right),I_{1}\left(g\right)\right\rangle _{L^{2}\left(\mu^{*}\right)}=\left\langle f,g\right\rangle _{L^{2}\left(\mu\right)},
\]

and
\[
\left\Vert I_{1}\left(f\right)\right\Vert _{L^{2}\left(\mu^{*}\right)}=\left\Vert f\right\Vert _{2}.
\]
\end{thm}

The above theorem is actually the ground for the following fundamental
result (see \cite{Ner1996CatInfGroups}, Chapter 10, Section 4):
\begin{thm}
\label{thm:Fock}There is a natural isometric identification of $L^{2}\left(\mu^{*}\right)$
as the Fock space $F\left(\text{\ensuremath{L^{2}\left(\mu\right)}}\right)$
over $L^{2}\left(\mu\right)$:
\[
L^{2}\left(\mu^{*}\right)\simeq F\left(\text{\ensuremath{L^{2}\left(\mu\right)}}\right):=\mathbb{C}\oplus L^{2}\left(\mu\right)\oplus L^{2}\left(\mu\right)^{\odot2}\oplus\cdots\oplus L^{2}\left(\mu\right)^{\odot n}\oplus\cdots
\]
\end{thm}

The latter space is understood as the completion of the infinite orthogonal
sum of the symmetric tensor products of $L^{2}\left(\mu\right)$ with
the appropriately scaled scalar product.

Let us succinctly describe how this identification is actually implemented:

Set, for any $n\ge1$, the \emph{stochastic integral of order $n$}
by the following formula:
\[
I_{n}\left(f^{\otimes n}\right):=\int_{\Delta_{n}^{c}}f^{\otimes n}d\left(N-\mu\right)^{\otimes n}\in L^{2}\left(\mu^{*}\right).
\]

where $\Delta_{n}$ is the subset of $X^{n}$ where at least two coordinates
coincide, $f$ is a simple function in $L^{2}\left(\mu\right)$ and
$f^{\otimes}$ its tensor product, an element of $L^{2}\left(\mu^{\otimes n}\right)_{\mid sym}\simeq L^{2}\left(\mu\right)^{\odot n}$.

The following formula holds for simple functions $f$ and $g$ in
$L^{2}\left(\mu\right)$:
\[
\left\langle I_{n}\left(f^{\otimes n}\right),I_{n}\left(g^{\otimes n}\right)\right\rangle _{L^{2}\left(\mu^{*}\right)}=n!\left\langle f^{\otimes n},g^{\otimes n}\right\rangle _{L^{2}\left(\mu\right)^{\odot n}}
\]

and if $n\neq m$,
\[
\left\langle I_{n}\left(f^{\otimes n}\right),I_{n}\left(g^{\otimes m}\right)\right\rangle _{L^{2}\left(\mu^{*}\right)}=0.
\]

The formula $I_{n}\left(f^{\otimes n}\right)$ extends isometrically
to any $f\in L^{2}\left(\mu\right)$ and if we set $\mathcal{I}_{2}^{\left(0\right)}=\mathbb{C}\cdot1_{X^{*}}$,
$\mathcal{I}_{2}^{\left(1\right)}=\mathcal{I}_{2}$ and $\mathcal{I}_{2}^{\left(n\right)}=\text{Span}\left\langle I_{n}\left(f^{\otimes n}\right),f\in L^{2}\left(\mu\right)\right\rangle $,
one gets the orthogonal sum:
\[
L^{2}\left(\mu^{*}\right)=\overline{\mathcal{I}_{2}^{\left(0\right)}\oplus\mathcal{I}_{2}^{\left(1\right)}\oplus\mathcal{I}_{2}^{\left(2\right)}\oplus\cdots\oplus\mathcal{I}_{2}^{\left(n\right)}\oplus\cdots}
\]

where each $\mathcal{I}_{2}^{\left(n\right)}$ is isometrically identified
to $\left(L^{2}\left(\mu\right)^{\odot n},n!\left\langle \cdot,\cdot\right\rangle _{L^{2}\left(\mu\right)^{\odot n}}\right)$
through $I_{n}$ and called the \emph{chaos of order $n$}.

Observe now that if $T$ is an endomorphism of $\left(X,\mathcal{A},\mu\right)$,
then, for any $f\in L^{2}\left(\mu\right)$:
\[
I_{n}\left(f\right)\circ T_{*}=I_{n}\left(\left(f\circ T\right)^{\otimes n}\right),
\]

in particular, $T_{*}$ preserves each chaos.

We end this section by the following lemma that we shall need later,
it is obvious given the above discussion:
\begin{lem}
\label{lem:projection}On $L^{2}\left(\mu^{*}\right)$, let $P$ be
the orthogonal projection on $\mathcal{I}_{2}$. Then $P$ and $T_{*}$
commute.
\end{lem}

\section{\label{sec:Integrable-stochastic-integrals}Integrable stochastic
integrals, the Orlicz space $L^{\Phi}\left(\mu\right)$}

The integrable case has deserved considerably less attention than
the square-integrable case, the latter being very handy to work with
given its Hilbertian nature.

Remark \ref{rem:ID Random variable} together with the integrability
criterion found in \cite{Sato99LevPro}, page 163 yield:
\begin{align}
I\left(f\right)\in L^{1}\left(\mu^{*}\right) & \Longleftrightarrow\int_{X}\min\left(\left|f\right|,\left|f\right|^{2}\right)d\mu<\infty.\label{eq:L^1 Stoch Int}
\end{align}

And it turns out that if this condition is satisfied, $\mathbb{E}_{\mu^{*}}\left[I\left(f\right)\right]=\mu\left(f1_{\left|f\right|>1}\right)$
and this allows to extend the definition of $I_{1}$ by setting
\[
I_{1}\left(f\right)=\lim_{\epsilon\to0}N\left(f1_{\left|f\right|>\epsilon}\right)-\mu\left(f1_{\left|f\right|>\epsilon}\right),
\]

in measure.

As a result, we get
\begin{equation}
I\left(f\right)\in L_{0}^{1}\left(\mu^{*}\right)\Longleftrightarrow\int_{X}\min\left(\left|f\right|,\left|f\right|^{2}\right)d\mu<\infty\;\text{and}\:I\left(f\right)=I_{1}\left(f\right).\label{eq:centered stoch int}
\end{equation}

We set $\mathcal{I}_{1}:=\mathcal{I}\cap L_{0}^{1}\left(\mu^{*}\right)$.

\subsection{\label{subsec:Orlicz}The Orlicz space $L^{\Phi}\left(\mu\right)$}

To understand the set $L^{\Phi}\left(\mu\right)$ of functions $f\in L^{0}\left(\mu\right)$
satisfying $\int_{X}\min\left(\left|f\right|,\left|f\right|^{2}\right)d\mu<\infty$,
it was natural (see \cite{Urbanik1967}) to replace it by the equivalent
condition
\[
\int_{X}\Phi\left(\left|f\right|\right)d\mu<\infty,
\]

$\Phi$ being the function on $\mathbb{R}_{+}$ given by:
\[
x\mapsto\begin{cases}
x^{2} & \text{if }x\le1\\
2x-1 & \text{if }x>1
\end{cases}
\]

The consequences are decisive, we list them thereafter, following
the first four chapters of \cite{Rao1991}:

$\Phi$ is a \emph{Young function}, that is
\begin{itemize}
\item $\Phi\left(0\right)=0$
\item $\lim_{x\to\infty}\Phi\left(x\right)=+\infty$
\item $\Phi$ is convex
\end{itemize}
Moreover, $\Phi$ is continuous, strictly increasing and \emph{$\Delta_{2}$-regular},
that is, for all $x\ge0$:
\[
\Phi\left(2x\right)\le4\Phi\left(x\right).
\]

A number of objects associated to such a function come naturally :

One defines $N_{\Phi}$ on $L^{\Phi}\left(\mu\right)$ by:

\[
N_{\Phi}\left(f\right):=\inf\left\{ \lambda>0,\,\int_{X}\Phi\left(\left|\frac{f}{\lambda}\right|\right)d\mu\le1\right\} .
\]

Also the \emph{complementary function} to $\Phi$ is the function
$\Psi$ defined on $\mathbb{R}_{+}$ by
\[
\Psi\left(y\right)=\sup\left\{ xy-\Phi\left(x\right)\right\} 
\]

and in this case, we have
\[
\Psi\left(y\right)=\begin{cases}
y^{2} & \text{if }y\le2\\
+\infty & \text{if }y>2
\end{cases}.
\]

This yields another set
\[
L^{\Psi}\left(\mu\right):=\left\{ g\in L^{0}\left(\mu\right),\exists\alpha>0\,\int_{X}\Psi\left(\left|\alpha g\right|\right)d\mu<\infty\right\} ,
\]

with $N_{\Psi}$ being defined similarly on $L^{\Psi}\left(\mu\right)$
by
\[
N_{\Psi}\left(g\right):=\inf\left\{ \lambda>0,\,\int_{X}\Psi\left(\left|\frac{g}{\lambda}\right|\right)d\mu\le1\right\} .
\]

The complementary function $\Psi$ allows to define the quantity $\left\Vert f\right\Vert _{\Phi}$
for $f\in L^{\Phi}\left(\mu\right)$ by:
\[
\left\Vert f\right\Vert _{\Phi}:=\sup\left\{ \int_{X}\left|fg\right|d\mu,\:\int_{X}\Psi\left(\left|g\right|\right)d\mu\le1\right\} .
\]

We can now sum up the properties of these objects, thanks to the well
understood theory of Orlicz spaces (see \cite{Rao1991} again) as
well as simple considerations:
\begin{thm}
\label{thm:Orlicz}The following holds:
\begin{itemize}
\item $L^{\Phi}\left(\mu\right)$ and $L^{\Psi}\left(\mu\right)$ are Orlicz
vector spaces.
\item $N_{\Phi}$ and $\left\Vert \cdot\right\Vert _{\Phi}$ are both norms
on $L^{\Phi}\left(\mu\right)$ ($N_{\Phi}$ is called the \emph{gauge}
or \emph{Luxemburg norm} and $\left\Vert \cdot\right\Vert _{\Phi}$
the \emph{Orlicz norm}).
\item For all $f\in L^{\Phi}\left(\mu\right)$,
\begin{equation}
N_{\Phi}\left(f\right)\le\left\Vert f\right\Vert _{\Phi}\le2N_{\Phi}\left(f\right).\label{eq:gauge-Orlicz-norms}
\end{equation}
\item $\left(L^{\Phi}\left(\mu\right),N_{\Phi}\right)$ is a non-reflexive
separable Banach space.
\item $L^{1}\left(\mu\right)\subset L^{\Phi}\left(\mu\right)$ and $\overline{L_{0}^{1}\left(\mu\right)}=\overline{L^{1}\left(\mu\right)}=L^{\Phi}\left(\mu\right)$.
\item $L^{2}\left(\mu\right)\subset L^{\Phi}\left(\mu\right)$ and $\overline{L^{2}\left(\mu\right)}=L^{\Phi}\left(\mu\right)$.
\item $L^{\Psi}\left(\mu\right)=L^{2}\left(\mu\right)\cap L^{\infty}\left(\mu\right).$
\item $N_{\Psi}=\max\left(\left\Vert \cdot\right\Vert _{2},\frac{1}{2}\left\Vert \cdot\right\Vert _{\infty}\right).$
\item For all $f\in L^{\Phi}\left(\mu\right)$, $\left\Vert f\right\Vert _{\Phi}:=\sup\left\{ \int_{X}\left|fg\right|d\mu,\:N_{\Psi}\left(g\right)\le1\right\} $.
\item $\left(L^{\Psi}\left(\mu\right),N_{\Psi}\right)$ is (a representation
of) the topological dual $\left(L^{\Phi}\left(\mu\right)^{\prime},\left\Vert \cdot\right\Vert _{\Phi}^{\prime}\right)$
of $\left(L^{\Phi}\left(\mu\right),\left\Vert \cdot\right\Vert _{\Phi}\right)$.
\item If $\left\{ f_{n}\right\} _{n\in\mathbb{N}}$ is a sequence of elements
of $L^{\Phi}\left(\mu\right)$ converging to $f\in L^{\Phi}\left(\mu\right)$
with respect to $\left\Vert \cdot\right\Vert _{\Phi}$, then there
exists an $\mu$-a.e. converging subsequence.
\end{itemize}
\end{thm}

\subsection{The Poisson-Orlicz norm and Urbanik's theorem}

Set $\mathcal{I}_{1}:=\mathcal{I}\cap L_{0}^{1}\left(\mu^{*}\right)$.
$I_{1}$ is an onto linear map between $L^{\Phi}\left(\mu\right)$
and $\mathcal{I}_{1}$, it is also one-to-one thanks to Remark \ref{rem:ID Random variable},
for example. This leads naturally to the definition:
\begin{defn}
For all $f\in L^{\Phi}\left(\mu\right)$,
\[
\left\Vert f\right\Vert _{*}:=\left\Vert I_{1}\left(f\right)\right\Vert _{L^{1}\left(\mu^{*}\right)},
\]

defines a norm. We call it the \emph{Poisson-Orlicz norm} of $L^{\Phi}\left(\mu\right)$.
\end{defn}

As such, $\left(\mathcal{I}_{1},\left\Vert \cdot\right\Vert _{L^{1}\left(\mu^{*}\right)}\right)$
is isometric to $\left(L^{\Phi}\left(\mu\right),\left\Vert \cdot\right\Vert _{*}\right)$
.

The natural question is therefore to compare this norm with $N_{\Phi}$
and $\left\Vert \cdot\right\Vert _{\Phi}$. The answer has been given
by Urbanik \cite{Urbanik1967}:
\begin{thm}
\label{thm:Urbanik}The norms $\left\Vert \cdot\right\Vert _{*}$,
$N_{\Phi}$ and $\left\Vert \cdot\right\Vert _{\Phi}$ are equivalent.
\end{thm}

In \cite{Marcus2001}, explicit bounds were obtained:
\begin{equation}
\left(0.125\right)N_{\Phi}\left(f\right)\le\left\Vert f\right\Vert _{*}\le\left(2.125\right)N_{\Phi}\left(f\right),\label{eq:Rosinski}
\end{equation}

which gives another proof of Urbanik's result. The purpose of the
authors in \cite{Marcus2001} was to a get good approximations of
the $L^{1}$-norm of integrable infinitely divisible random vectors
(which can always been obtained as a particular (multidimensional)
stochastic integral).

For the sake of completeness and in an effort to better understand
this space, we will give our own short proof of Theorem \ref{thm:Urbanik}.
This will allow us to get an optimal upper bound with respect to $\left\Vert \cdot\right\Vert _{\Phi}$
(and thus improve the right-hand part of Eq (\ref{eq:Rosinski}):
$\left\Vert f\right\Vert _{*}\le2N_{\Phi}\left(f\right)$).

Let us see how $\left\Vert \cdot\right\Vert _{*}$ behaves:
\begin{prop}
\label{prop:obsPoisson}We have
\begin{itemize}
\item For any $f\in L^{1}\left(\mu\right)$, 
\[
\left\Vert f\right\Vert _{*}\le2\left\Vert f\right\Vert _{1}.
\]
\item For any $f\in L^{2}\left(\mu\right)$,
\[
\left\Vert f\right\Vert _{*}\le\left\Vert f\right\Vert _{2}.
\]
\item For any $f\in L^{\Phi}\left(\mu\right)$,
\[
\left\Vert f\right\Vert _{*}\le\left\Vert f\right\Vert _{\Phi}.
\]
\end{itemize}
Moreover, $\left\Vert \cdot\right\Vert _{*}$ and $\left\Vert \cdot\right\Vert _{\Phi}$
are different and the bound is optimal.
\end{prop}

\begin{proof}
If $f\in L^{1}\left(\mu\right)$, then
\begin{align*}
\left\Vert f\right\Vert _{*} & =\left\Vert I_{1}\left(f\right)\right\Vert _{L^{1}\left(\mu^{*}\right)}\\
 & =\mathbb{E}_{\mu^{*}}\left[\left|I_{1}\left(f\right)\right|\right]\\
 & =\mathbb{E}_{\mu^{*}}\left[\left|N\left(f\right)-\mu\left(f\right)\right|\right]\\
 & \le\mathbb{E}_{\mu^{*}}\left[\left|N\left(f\right)\right|\right]+\left|\mu\left(f\right)\right|\\
 & \le\mathbb{E}_{\mu^{*}}\left[N\left(\left|f\right|\right)\right]+\mu\left(\left|f\right|\right)\\
 & =2\left\Vert f\right\Vert _{1}.
\end{align*}

And if $f\in L^{2}\left(\mu\right)$,
\begin{align*}
\left\Vert f\right\Vert _{*} & =\left\Vert I_{1}\left(f\right)\right\Vert _{L^{1}\left(\mu^{*}\right)}\\
 & \le\left\Vert I_{1}\left(f\right)\right\Vert _{L^{2}\left(\mu^{*}\right)}\\
 & =\left\Vert f\right\Vert _{2},
\end{align*}

thanks to Theorem \ref{thm:L^2 isomorphism}.

Denote by $\left(L^{\Phi}\left(\mu\right)^{\prime*},\left\Vert \cdot\right\Vert _{*}^{\prime}\right)$
the topological dual of $\left(L^{\Phi}\left(\mu\right),\left\Vert \cdot\right\Vert _{*}\right)$
and let $\phi\in L^{\Phi}\left(\mu\right)^{\prime*}$. From the beginning
of the proof, for any $f\in L^{1}\left(\mu\right)$:
\begin{equation}
\left|\phi\left(f\right)\right|\le\left\Vert \phi\right\Vert _{*}^{\prime}\left\Vert f\right\Vert _{*}\le2\left\Vert \phi\right\Vert _{*}^{\prime}\left\Vert f\right\Vert _{1},\label{eq:L^1Phi}
\end{equation}

and for any $f\in L^{2}\left(\mu\right)$:
\begin{equation}
\left|\phi\left(f\right)\right|\le\left\Vert \phi\right\Vert _{*}^{\prime}\left\Vert f\right\Vert _{*}\le\left\Vert \phi\right\Vert _{*}^{\prime}\left\Vert f\right\Vert _{2}.\label{eq:L^2Phi}
\end{equation}

Thus $\phi$ is a continuous linear form on both $\left(L^{1}\left(\mu\right),\left\Vert \cdot\right\Vert _{1}\right)$
and $\left(L^{2}\left(\mu\right),\left\Vert \cdot\right\Vert _{2}\right)$,
and there exists a unique $g_{1}\in L^{\infty}\left(\mu\right)$ such
that, for any $f\in L^{1}\left(\mu\right)$:

\[
\phi\left(f\right)=\int_{X}g_{1}fd\mu,
\]

and a unique $g_{2}\in L^{2}\left(\mu\right)$ such that, for any
$f\in L^{2}\left(\mu\right)$:

\[
\phi\left(f\right)=\int_{X}g_{2}fd\mu.
\]

Of course $g_{1}=g_{2}$ $\mu$-a.e..

Let $g:=g_{1}\in L^{2}\left(\mu\right)\cap L^{\infty}\left(\mu\right)$
and take $f$ in $L^{\Phi}\left(\mu\right)$. We have $f=f1_{\left|f\right|\le1}+f1_{\left|f\right|>1}$
with $f1_{\left|f\right|\le1}\in L^{2}\left(\mu\right)$ and $f1_{\left|f\right|>1}\in L^{1}\left(\mu\right)$
therefore
\begin{align*}
\phi\left(f\right) & =\phi\left(f1_{\left|f\right|\le1}\right)+\phi\left(f1_{\left|f\right|>1}\right)\\
 & =\int_{X}gf1_{\left|f\right|\le1}d\mu+\int_{X}gf1_{\left|f\right|>1}d\mu\\
 & =\int_{X}gfd\mu.
\end{align*}

In particular, $\phi\in L^{\Phi}\left(\mu\right)^{\prime}$, which
yields $L^{\Phi}\left(\mu\right)^{\prime*}\subset L^{\Phi}\left(\mu\right)^{\prime}\simeq L^{2}\left(\mu\right)\cap L^{\infty}\left(\mu\right)$.
From now on, we identify $L^{\Phi}\left(\mu\right)^{\prime*}$ as
a subspace of $L^{2}\left(\mu\right)\cap L^{\infty}\left(\mu\right)\simeq L^{\Phi}\left(\mu\right)^{\prime}$.

From the inequalities (\ref{eq:L^1Phi}) and (\ref{eq:L^2Phi}), for
any $g\in L^{\Phi}\left(\mu\right)^{\prime*}$,
\[
\max\left(\left\Vert g\right\Vert _{2},\frac{1}{2}\left\Vert g\right\Vert _{\infty}\right)\le\left\Vert g\right\Vert _{*}^{\prime},
\]

thus, for any $g\in L^{\Phi}\left(\mu\right)^{\prime*}$,
\[
N_{\Psi}\left(g\right)\le\left\Vert g\right\Vert _{*}^{\prime}.
\]

Hence for any $f\in L^{\Phi}\left(\mu\right)$:
\begin{align*}
\left\Vert f\right\Vert _{*} & =\sup_{g\in L^{\Phi}\left(\mu\right)^{\prime*}}\frac{\int_{X}gfd\mu}{\left\Vert g\right\Vert _{*}^{\prime}}\\
 & \le\sup_{g\in L^{\Phi}\left(\mu\right)^{\prime*}}\frac{\int_{X}gfd\mu}{N_{\Psi}\left(g\right)}\\
 & \le\sup_{g\in L^{2}\left(\mu\right)\cap L^{\infty}\left(\mu\right)}\frac{\int_{X}gfd\mu}{N_{\Psi}\left(g\right)}\\
 & =\left\Vert f\right\Vert _{\Phi}.
\end{align*}

Now observe that for any $f\in L^{\Phi}\left(\mu\right)$, we have
\[
\left\Vert \left|f\right|\right\Vert _{\Phi}=\left\Vert f\right\Vert _{\Phi}.
\]

This is not the case with $\left\Vert \cdot\right\Vert _{*}$. To
see this, let $A_{n}$ and $B_{n}$, $n\ge2$ be measurable sets such
that, for all $n\ge2$, $A_{n}\cap B_{n}=\emptyset$ and $\mu\left(A_{n}\right)=\mu\left(B_{n}\right)=\frac{1}{2n}$,
and set $f_{n}:=1_{A_{n}}-1_{B_{n}}$, so that $f_{n}\in L_{0}^{1}\left(\mu\right)$
and $\left\Vert f_{n}\right\Vert _{1}=\frac{1}{n}$.

We have:

\begin{align*}
\left\Vert \left|f_{n}\right|\right\Vert _{*} & =\left\Vert \left|1_{A_{n}\cup B_{n}}\right|\right\Vert _{*}\\
 & =\mathbb{E}_{\mu^{*}}\left[\left|N\left(A_{n}\cup B_{n}\right)-\mu\left(A_{n}\cup B_{n}\right)\right|\right]\\
 & =\mathbb{E}_{\mu^{*}}\left[\mu\left(A_{n}\cup B_{n}\right)1_{N\left(A_{n}\cup B_{n}\right)=0}\right]\\
 & +\mathbb{E}_{\mu^{*}}\left[\left(N\left(A_{n}\cup B_{n}\right)-\mu\left(A_{n}\cup B_{n}\right)\right)1_{N\left(A_{n}\cup B_{n}\right)\ge1}\right]\\
 & =\frac{1}{n}e^{-\frac{1}{n}}+\mathbb{E}_{\mu^{*}}\left[N\left(A_{n}\cup B_{n}\right)-\mu\left(A_{n}\cup B_{n}\right)\right]\\
 & -\mathbb{E}_{\mu^{*}}\left[\left(N\left(A_{n}\cup B_{n}\right)-\mu\left(A_{n}\cup B_{n}\right)\right)1_{N\left(A_{n}\cup B_{n}\right)=0}\right]\\
 & =\frac{2}{n}e^{-\frac{1}{n}},
\end{align*}

whereas
\begin{align*}
\left\Vert f_{n}\right\Vert _{*} & =\mathbb{E}_{\mu^{*}}\left[\left|I_{1}\left(f_{n}\right)\right|\right]\\
 & =\mathbb{E}_{\mu^{*}}\left[\left|N\left(f_{n}\right)-\mu\left(f_{n}\right)\right|\right]\\
 & =\mathbb{E}_{\mu^{*}}\left[\left|N\left(f_{n}\right)\right|\right]\\
 & \le\mathbb{E}_{\mu^{*}}\left[N\left(\left|f_{n}\right|\right)\right]\\
 & =\left\Vert f_{n}\right\Vert _{1}=\frac{1}{n}.
\end{align*}

To prove the optimality, observe that, since for all $g\in L^{2}\left(\mu\right)\cap L^{\infty}\left(\mu\right)$,
$N_{\Psi}\left(g\right)=\max\left(\left\Vert g\right\Vert _{2},\frac{1}{2}\left\Vert g\right\Vert _{\infty}\right)\ge\frac{1}{2}\left\Vert g\right\Vert _{\infty}$,
we have, for any $f\in L^{1}\left(\mu\right)$:
\begin{align*}
\text{\ensuremath{\left\Vert f\right\Vert _{\Phi}}} & =\sup_{g\in L^{2}\left(\mu\right)\cap L^{\infty}\left(\mu\right)}\frac{\int_{X}gfd\mu}{N_{\Psi}\left(g\right)}\\
 & \le2\sup_{g\in L^{2}\left(\mu\right)\cap L^{\infty}\left(\mu\right)}\frac{\int_{X}gfd\mu}{\left\Vert g\right\Vert _{\infty}}\\
 & \le2\sup_{g\in L^{\infty}\left(\mu\right)}\frac{\int_{X}gfd\mu}{\left\Vert g\right\Vert _{\infty}}\\
 & =2\text{\ensuremath{\left\Vert f\right\Vert _{1}}}.
\end{align*}

Thus $\text{\ensuremath{\left\Vert \left|f_{n}\right|\right\Vert _{\Phi}}\ensuremath{\ensuremath{\le\frac{2}{n}}}}$.
\end{proof}
We have just seen that, in general, $\left\Vert f\right\Vert _{*}=\left\Vert \left|f\right|\right\Vert _{*}$
doesn't hold. Nevertheless we can compare the two quantities:
\begin{prop}
Let $f\in L^{\Phi}\left(\mu\right)$, real valued. Then
\[
\left\Vert f\right\Vert _{*}\le2\left\Vert \left|f\right|\right\Vert _{*},
\]

and 
\[
\left\Vert \left|f\right|\right\Vert _{*}\le2\left\Vert f\right\Vert _{*}.
\]
\end{prop}

\begin{proof}
Write $f=f^{+}-f^{-}$. Since $f^{+}$ and $f^{-}$ are supported
on disjoint sets, $I_{1}\left(f^{+}\right)$ and $I_{1}\left(f^{-}\right)$
are independent thanks to the independence properties of a Poisson
measure. Therefore, taking the conditional expectation with respect
to $\sigma\left(I_{1}\left(f^{-}\right)\right)$:
\begin{align*}
\mathbb{E}_{\mu^{*}}\left[\left|I_{1}\left(f\right)\right|\right] & =\mathbb{E}_{\mu^{*}}\left[\mathbb{E}_{\mu^{*}}\left[\left|I_{1}\left(f^{+}\right)-I_{1}\left(f^{-}\right)\right|\mid\sigma\left(I_{1}\left(f^{-}\right)\right)\right]\right]\\
 & \ge\mathbb{E}_{\mu^{*}}\left[\left|\mathbb{E}_{\mu^{*}}\left[I_{1}\left(f^{+}\right)-I_{1}\left(f^{-}\right)\mid\sigma\left(I_{1}\left(f^{-}\right)\right)\right]\right|\right]\\
 & =\mathbb{E}_{\mu^{*}}\left[\left|\mathbb{E}_{\mu^{*}}\left[I_{1}\left(f^{+}\right)\mid\sigma\left(I_{1}\left(f^{-}\right)\right)\right]-\mathbb{E}_{\mu^{*}}\left[I_{1}\left(f^{-}\right)\mid\sigma\left(I_{1}\left(f^{-}\right)\right)\right]\right|\right]\\
 & =\mathbb{E}_{\mu^{*}}\left[\left|\mathbb{E}_{\mu^{*}}\left[I_{1}\left(f^{+}\right)\right]-I_{1}\left(f^{-}\right)\right|\right]\\
 & =\mathbb{E}_{\mu^{*}}\left[\left|I_{1}\left(f^{-}\right)\right|\right].
\end{align*}

Thus $\left\Vert f\right\Vert _{*}\ge\left\Vert f^{-}\right\Vert _{*}$
and of course $\left\Vert f\right\Vert _{*}\ge\left\Vert f^{+}\right\Vert _{*}$.
This yields
\[
\max\left(\left\Vert f^{-}\right\Vert _{*},\left\Vert f^{+}\right\Vert _{*}\right)\le\left\Vert f\right\Vert _{*},
\]

and the triangular inequality implies
\[
\left\Vert f\right\Vert _{*}\le\left\Vert f^{+}\right\Vert _{*}+\left\Vert f^{-}\right\Vert _{*}\le2\max\left(\left\Vert f^{-}\right\Vert _{*},\left\Vert f^{+}\right\Vert _{*}\right),
\]

Hence
\[
\max\left(\left\Vert f^{-}\right\Vert _{*},\left\Vert f^{+}\right\Vert _{*}\right)\le\left\Vert f\right\Vert _{*}\le2\max\left(\left\Vert f^{-}\right\Vert _{*},\left\Vert f^{+}\right\Vert _{*}\right).
\]

Writing now $\left|f\right|=f^{+}+f^{-}$, the same reasoning gives
\[
\max\left(\left\Vert f^{-}\right\Vert _{*},\left\Vert f^{+}\right\Vert _{*}\right)\le\left\Vert \left|f\right|\right\Vert _{*}\le2\max\left(\left\Vert f^{-}\right\Vert _{*},\left\Vert f^{+}\right\Vert _{*}\right),
\]

and we get the result.
\end{proof}
We can now end this section with our proof of Urbanik's result, Theorem
\ref{thm:Urbanik}:
\begin{proof}
\emph{(of Theorem \ref{thm:Urbanik})} From Proposition \ref{prop:closnessofI},
it follows that $\mathcal{I}_{1}$ is closed in the Banach space $\left(L^{1}\left(\mu^{*}\right),\left\Vert \cdot\right\Vert _{L^{1}\left(\mu^{*}\right)}\right)$.
But $\left(L^{\Phi}\left(\mu\right),\left\Vert \cdot\right\Vert _{*}\right)$
is isometrically isomorphic to $\left(\mathcal{I}_{1},\left\Vert \cdot\right\Vert _{L^{1}\left(\mu^{*}\right)}\right)$
therefore $\left(L^{\Phi}\left(\mu\right),\left\Vert \cdot\right\Vert _{*}\right)$
is a Banach space. Thanks to Proposition \ref{prop:obsPoisson} and
the open map theorem applied to the identity on $L^{\Phi}\left(\mu\right)$,
the norms $\left\Vert \cdot\right\Vert _{*}$ and $\left\Vert \cdot\right\Vert _{\Phi}$
are equivalent.
\end{proof}

\section{\label{sec:-ergodic-results}$L^{\Phi}\left(\mu\right)$-ergodic
results}

\subsection{Mean ergodic theorem in $L^{\Phi}\left(\mu\right)$}

We now introduce an endomorphism $T$ on $\left(X,\mathcal{A},\mu\right)$.
The next observation follows from Remark \ref{rem:Almostlinear}.
\begin{prop}
\label{prop:Isom}$T_{*}$ satisfies, for any $f\in L^{\Phi}\left(\mu\right)$:
\[
I_{1}\left(f\right)\circ T_{*}=I_{1}\left(f\circ T\right).
\]
In particular, $T$ is an isometry of $\left(L^{\Phi}\left(\mu\right),\left\Vert \cdot\right\Vert _{*}\right)$.
\end{prop}

We can state:
\begin{thm}
\label{thm:ergodicL12}The dynamical system $\left(X,\mathcal{A},\mu,T\right)$
has no absolutely continuous $T$-invariant probability measure if
and only if, for every $f\in L^{1}\left(\mu\right)$ (or $f\in L^{\Phi}\left(\mu\right)$):
\[
\frac{1}{n}\sum_{k=1}^{n}f\circ T^{k}\to_{\left\Vert \cdot\right\Vert _{*}}0,
\]

as $n$ tends to $+\infty$.
\end{thm}

\begin{proof}
Assume $T$ has  no absolutely continuous $T$-invariant probability
measure (in particular $\left(X^{*},\mathcal{A}^{*},\mu^{*},T_{*}\right)$
is ergodic by Theorem \ref{thm:Ergodic basics}) and let $f\in L^{1}\left(\mu\right)$:
\begin{align*}
\left\Vert \frac{1}{n}\sum_{k=1}^{n}f\circ T^{k}\right\Vert _{*} & =\left\Vert I_{1}\left(\frac{1}{n}\sum_{k=1}^{n}f\circ T^{k}\right)\right\Vert _{L^{1}\left(\mu^{*}\right)}\\
 & =\left\Vert \frac{1}{n}\sum_{k=1}^{n}I_{1}\left(f\right)\circ T_{*}^{k}\right\Vert _{L^{1}\left(\mu^{*}\right)}\\
 & \to0
\end{align*}

as $n$ tends to infinity thanks to Birkhoff's ergodic theorem.

Let now $f\in L^{1}\left(\mu\right)$ and $g\in L^{\Phi}\left(\mu\right)$:
\begin{align*}
\left\Vert \frac{1}{n}\sum_{k=1}^{n}g\circ T^{k}\right\Vert _{*} & =\left\Vert \frac{1}{n}\sum_{k=1}^{n}\left(g-f\right)\circ T^{k}+\frac{1}{n}\sum_{k=1}^{n}f\circ T^{k}\right\Vert _{*}\\
 & \le\frac{1}{n}\sum_{k=1}^{n}\left\Vert \left(g-f\right)\circ T^{k}\right\Vert _{*}+\left\Vert \frac{1}{n}\sum_{k=1}^{n}f\circ T^{k}\right\Vert _{*}\\
 & =\left\Vert g-f\right\Vert _{*}+\left\Vert \frac{1}{n}\sum_{k=1}^{n}f\circ T^{k}\right\Vert _{*}
\end{align*}

since $T$ is an isometry of $L^{\Phi}\left(\mu\right)$ (Proposition
\ref{prop:Isom}).

Thus
\[
\limsup\left\Vert \frac{1}{n}\sum_{k=1}^{n}g\circ T^{k}\right\Vert _{*}\le\left\Vert g-f\right\Vert _{*}
\]

And the result follows by density of $L^{1}\left(\mu\right)$ in $L^{\Phi}\left(\mu\right)$
(Theorem \ref{thm:Orlicz}).

Conversely, assume there exists an absolutely continuous $T$-invariant
probability measure $\nu\ll\mu$. This implies the existence of a
non-trivial $T$-invariant set $A\in\mathcal{A}_{f}$ and we get:
\[
\left\Vert \frac{1}{n}\sum_{k=1}^{n}1_{A}\circ T^{k}\right\Vert _{*}=\left\Vert 1_{A}\right\Vert _{*}>0.
\]
\end{proof}
\begin{rem}
Mean ergodic theorems are often associated to pointwise convergence.
In our context, the result is immediate, for any $f\in L^{\Phi}\left(\mu\right)$:
\[
\frac{1}{n}\sum_{k=1}^{n}f\left(T^{k}x\right)
\]

converges as $n$ tends to infinity, for $\mu$-almost every $x\in X$
since, as we saw above, we can write $f=f1_{\left|f\right|\le1}+f1_{\left|f\right|>1}$
with $f1_{\left|f\right|\le1}\in L^{2}\left(\mu\right)$ and $f1_{\left|f\right|>1}\in L^{1}\left(\mu\right)$
and the $\mu$-almost sure convergence holds for both $L^{1}\left(\mu\right)$
(\cite{Hopf1937}) and $L^{2}\left(\mu\right)$ (\cite{Akcoglu1975}).
In particular, if $T$ has  no absolutely continuous $T$-invariant
probability measure, then for $\mu$-almost every $x\in X$:
\[
\frac{1}{n}\sum_{k=1}^{n}f\left(T^{k}x\right)\to0,
\]

as $n$ tends to infinity.
\end{rem}

\subsection{Blum-Hanson type Theorem}
\begin{thm}
\label{thm:Blum-Hanson_type}The following assertions are equivalent:
\begin{enumerate}
\item $\left(X,\mathcal{A},\mu,T\right)$ is of zero type
\item $\left\{ T^{n}\right\} _{n\in\mathbb{N}}$ tends weakly to $0$ in
$L^{\Phi}\left(\mu\right)$
\item for every strictly increasing sequence $\left\{ n_{k}\right\} _{k\in\mathbb{N}}$
of integers and every $f\in L^{1}\left(\mu\right)$ (or equivalently
$f\in L^{\Phi}\left(\mu\right)$):
\[
\frac{1}{n}\sum_{k=1}^{n}f\circ T^{n_{k}}\to_{\left\Vert \cdot\right\Vert _{*}}0,
\]
\\
as $n$ tends to infinity.
\end{enumerate}
\end{thm}

\begin{proof}
Observe that $\left(X,\mathcal{A},\mu,T\right)$ is of zero type is
equivalent to the fact that $\left\{ T^{n}\right\} _{n\in\mathbb{N}}$
tends weakly to $0$ in $L^{2}\left(\mu\right)$. But since $L^{\Phi}\left(\mu\right)^{\prime}\simeq L^{2}\left(\mu\right)\cap L^{\infty}\left(\mu\right)$
and $\overline{L^{2}\left(\mu\right)}=L^{\Phi}\left(\mu\right)$ (Theorem
\ref{thm:Orlicz}), we easily obtain that zero type is equivalent
to the fact that $\left\{ T^{n}\right\} _{n\in\mathbb{N}}$ tends
weakly to $0$ in $L^{\Phi}\left(\mu\right)$.

If $\left(X,\mathcal{A},\mu,T\right)$ is of zero type then $\left(X^{*},\mathcal{A}^{*},\mu^{*},T_{*}\right)$
is mixing (Theorem \ref{thm:Ergodic basics}) and Blum-Hanson theorem
(\cite{Blum1960}) applies: for every strictly increasing sequence
$\left\{ n_{k}\right\} _{k\in\mathbb{N}}$ of integers, every $f\in L^{\Phi}\left(\mu\right)$,

\[
\frac{1}{n}\sum_{k=1}^{n}I_{1}\left(f\right)\circ T_{*}^{n_{k}}\to_{\left\Vert \cdot\right\Vert _{L^{1}\left(\mu^{*}\right)}}0,
\]
but
\[
\left\Vert \frac{1}{n}\sum_{k=1}^{n}I_{1}\left(f\right)\circ T_{*}^{n_{k}}\right\Vert _{L^{1}\left(\mu^{*}\right)}=\left\Vert \frac{1}{n}\sum_{k=1}^{n}f\circ T^{n_{k}}\right\Vert _{*},
\]

that is,
\[
\frac{1}{n}\sum_{k=1}^{n}f\circ T^{n_{k}}\to_{\left\Vert \cdot\right\Vert _{*}}0.
\]

Conversely, assume the convergence holds for every $f\in L^{1}\left(\mu\right)$
and every strictly increasing sequence $\left\{ n_{k}\right\} _{k\in\mathbb{N}}$
of integers and consider sets $A$ and $B$ in $\mathcal{A}_{f}$.

We want to prove $\mu\left(A\cap T^{-n}B\right)\to0$ as $n$ tends
to infinity. By a standard exercise, it is equivalent to show the
Ces\`aro convergence to zero of its subsequences i.e., for every strictly
increasing sequence $\left\{ n_{k}\right\} _{k\in\mathbb{N}}$:
\[
\frac{1}{n}\sum_{k=1}^{n}\mu\left(A\cap T^{-n_{k}}B\right)\to0,
\]

as $n$ tends to infinity, but this will indeed be the case, since
\begin{align*}
\left|\frac{1}{n}\sum_{k=1}^{n}\mu\left(A\cap T^{-n_{k}}B\right)\right| & =\left|\int_{X}1_{A}\cdot\frac{1}{n}\sum_{k=1}^{n}1_{B}\circ T^{n_{k}}d\mu\right|\\
 & \le\left\Vert 1_{A}\right\Vert _{*}^{\prime}\left\Vert \frac{1}{n}\sum_{k=1}^{n}1_{B}\circ T^{n_{k}}\right\Vert _{*}
\end{align*}

as $1_{A}\in L^{\infty}\left(\mu\right)\cap L^{2}\left(\mu\right)$
and $1_{B}\in L^{1}\left(\mu\right)\subset L^{\Phi}\left(\mu\right)$.

Hence $T$ is of zero type.
\end{proof}

\section{\label{sec:Duality-and-transfer}Duality and transfer operators}

\subsection{Duality}

It will be useful to have another representation of $\left(L^{2}\left(\mu\right)\cap L^{\infty}\left(\mu\right),\left\Vert \cdot\right\Vert _{*}^{\prime}\right)$,
the dual of $\left(L^{\Phi}\left(\mu\right),\left\Vert \cdot\right\Vert _{*}\right)$,
and that requires the introduction of a few objects.

Recall that, we denoted by $\mathcal{I}_{2}:=\left\{ I_{1}\left(f\right),\:f\in L^{2}\left(\mu\right)\right\} $
the first chaos and we pointed out the fact that $I_{1}$ is implementing
an isometric isomorphism between $\left(\mathcal{I}_{2},\left\Vert \cdot\right\Vert _{L^{2}\left(\mu^{*}\right)}\right)$
and $\left(L^{2}\left(\mu\right),\left\Vert \cdot\right\Vert _{2}\right)$.
Recall that we have introduced $P$ in Lemma \ref{lem:projection}
as the orthogonal projection on $\mathcal{I}_{2}$; we can then define
$\pi$ from $L^{2}\left(\mu^{*}\right)$ to $L^{2}\left(\mu\right)$
by
\[
\pi:=I_{1}^{-1}\circ P.
\]

Consider now the restriction of $\pi$ to $L^{\infty}\left(\mu^{*}\right)\subset L^{2}\left(\mu^{*}\right)$.
We will show:
\begin{prop}
$\pi$ is a surjective bounded operator from $\left(L^{\infty}\left(\mu^{*}\right),\left\Vert \cdot\right\Vert _{L^{\infty}\left(\mu^{*}\right)}\right)$
to $\left(L^{2}\left(\mu\right)\cap L^{\infty}\left(\mu\right),\left\Vert \cdot\right\Vert _{*}^{\prime}\right)$.
\end{prop}

\begin{proof}
If $Z\in L^{\infty}\left(\mu^{*}\right)$, then the formula
\[
L^{\Phi}\left(\mu\right)\ni f\mapsto\mathbb{E}_{\mu^{*}}\left[ZI_{1}\left(f\right)\right]
\]

defines a continuous linear form on $\left(L^{\Phi}\left(\mu\right),\left\Vert \cdot\right\Vert _{*}\right)$
since
\begin{align}
\left|\mathbb{E}_{\mu^{*}}\left[ZI_{1}\left(f\right)\right]\right| & \le\left\Vert Z\right\Vert _{L^{\infty}\left(\mu^{*}\right)}\left\Vert I_{1}\left(f\right)\right\Vert _{L^{1}\left(\mu^{*}\right)}\nonumber \\
 & =\left\Vert Z\right\Vert _{L^{\infty}\left(\mu^{*}\right)}\left\Vert f\right\Vert _{*}.\label{eq:continuity=00005Cpi}
\end{align}

Therefore, there exists $g\in L^{2}\left(\mu\right)\cap L^{\infty}\left(\mu\right)$
such that, for any $f\in L^{\Phi}\left(\mu\right)$:
\[
\mathbb{E}_{\mu^{*}}\left[ZI_{1}\left(f\right)\right]=\int_{X}gfd\mu
\]

but, if $f\in L^{2}\left(\mu\right)$, then, if $P\left(Z\right)=I_{1}\left(h\right)$
for some $h\in L^{2}\left(\mu\right)$, then
\begin{align*}
\mathbb{E}_{\mu^{*}}\left[ZI_{1}\left(f\right)\right] & =\mathbb{E}_{\mu^{*}}\left[I_{1}\left(g\right)I_{1}\left(f\right)\right]\\
 & =\int_{X}hfd\mu.
\end{align*}

Therefore $h=g$ and $\pi\left(Z\right)=g\in L^{2}\left(\mu\right)\cap L^{\infty}\left(\mu\right)$.

Conversely, for any $g\in L^{2}\left(\mu\right)\cap L^{\infty}\left(\mu\right)$,
the map
\[
F\mapsto\int_{X}gI_{1}^{-1}\left(F\right)d\mu=\mathbb{E}_{\mu^{*}}\left[I_{1}\left(g\right)F\right]
\]

defines a continuous linear form on $\left(\mathcal{I}_{1},\left\Vert \cdot\right\Vert _{L^{1}\left(\mu^{*}\right)}\right)$
that extends, thanks to Hahn-Banach theorem to $\left(L^{1}\left(\mu^{*}\right),\left\Vert \cdot\right\Vert _{L^{1}\left(\mu^{*}\right)}\right)$,
that is, there exists $Z\in L^{\infty}\left(\mu^{*}\right)$ such
that, for, any $F\in L^{1}\left(\mu^{*}\right)$:
\[
\mathbb{E}_{\mu^{*}}\left[ZF\right]=\mathbb{E}_{\mu^{*}}\left[I_{1}\left(g\right)F\right].
\]

In particular, for any $f\in\text{\ensuremath{L^{2}\left(\mu\right)}}$,
\[
\mathbb{E}_{\mu^{*}}\left[ZI_{1}\left(f\right)\right]=\mathbb{E}_{\mu^{*}}\left[I_{1}\left(g\right)I_{1}\left(f\right)\right].
\]

And this proves that $\pi\left(Z\right)=g$. Therefore the range of
$\pi$ is $L^{2}\left(\mu\right)\cap L^{\infty}\left(\mu\right)$
and from (\ref{eq:continuity=00005Cpi}), we get
\[
\left\Vert \pi\left(Z\right)\right\Vert _{*}^{\prime}\le\left\Vert Z\right\Vert _{L^{\infty}\left(\mu^{*}\right)}.
\]
\end{proof}
\begin{prop}
\label{prop:dualitystructure}In the duality $\sigma\left(L^{1}\left(\mu^{*}\right),L^{\infty}\left(\mu^{*}\right)\right)$,
the orthogonal $\mathcal{I}_{1}^{\perp}$ of $\mathcal{I}_{1}$ is
the subspace $\ker\pi$. In particular, there is a canonical isometric
identification
\[
\left(L^{\infty}\left(\mu^{*}\right)/\ker\pi,\left\Vert \cdot\right\Vert _{/}\right)\simeq\left(L^{2}\left(\mu\right)\cap L^{\infty}\left(\mu\right),\left\Vert \cdot\right\Vert _{*}^{\prime}\right),
\]

where $\left\Vert \cdot\right\Vert _{/}$ is the usual quotient norm
induced by $\left\Vert \cdot\right\Vert _{L^{\infty}\left(\mu^{*}\right)}$.
\end{prop}

\begin{proof}
If $f\in L^{2}\left(\mu\right)$ and $Z\in L^{\infty}\left(\mu^{*}\right)$,
then
\begin{align*}
\mathbb{E}_{\mu^{*}}\left[ZI_{1}\left(f\right)\right] & =\mathbb{E}_{\mu^{*}}\left[I_{1}\left(\pi\left(Z\right)\right)I_{1}\left(f\right)\right]\\
 & =\int_{X}\pi\left(Z\right)fd\mu.
\end{align*}

This formula extends by density of $L^{2}\left(\mu\right)$ in $L^{\Phi}\left(\mu\right)$
with respect to $\left\Vert \cdot\right\Vert _{*}$, and thus of $\mathcal{I}_{2}$
in $\mathcal{I}_{1}$ with respect to $\left\Vert \cdot\right\Vert _{L^{1}\left(\mu^{*}\right)}$.

Therefore $\mathcal{I}_{1}^{\perp}=\ker\pi$.

The map $\pi$ induces an isometric isomorphism between $\left(L^{\infty}\left(\mu^{*}\right)/\ker\pi,\left\Vert \cdot\right\Vert _{/}\right)$
and $\left(L^{2}\left(\mu\right)\cap L^{\infty}\left(\mu\right),\left\Vert \cdot\right\Vert _{*}^{\prime}\right)$.
\end{proof}
\begin{rem}
We could have appeal to general results on Banach spaces to get that
$L^{\infty}\left(\mu^{*}\right)/\mathcal{I}_{1}^{\perp}$ is canonically
isomorphic (and isometric) to the dual of $\mathcal{I}_{1}$, however
we get an explicit description through the operator $\pi$ as above.
\end{rem}

\subsection{Transfer operators}

Let $T$ be an endomorphism on $\left(X,\mathcal{A},\mu\right)$.
The \emph{transfer operator} is the operator $\widehat{T}$ acting
on $L^{1}\left(\mu\right)$ defined as the predual operator of $T$
seen as acting as an isometry on $L^{\infty}\left(\mu\right)$, characterized
by the following relation, for any $f\in L^{1}\left(\mu\right)$ and
$g\in L^{\infty}\left(\mu\right)$:
\[
\int_{X}\widehat{T}f\cdot gd\mu=\int_{X}f\cdot g\circ Td\mu.
\]

Similarly, the same object exists at the level of the Poisson suspension
$\left(X^{*},\mathcal{A}^{*},\mu^{*},T_{*}\right)$: $\widehat{T_{*}}$
is the transfer operator of $T_{*}$ .

We will mimic the same procedure to define a ``transfer'' operator
acting on $L^{\Phi}\left(\mu\right)$.
\begin{prop}
\label{prop:Tisomdual}The isometry $T_{*}$ on $\left(L^{\infty}\left(\mu^{*}\right),\left\Vert \cdot\right\Vert _{L^{\infty}\left(\mu^{*}\right)}\right)$
acts also on the quotient space $\left(L^{\infty}\left(\mu^{*}\right)/\ker\pi,\left\Vert \cdot\right\Vert _{/}\right)$
as an isometry and corresponds to $T$ through the identification
with $\left(L^{2}\left(\mu\right)\cap L^{\infty}\left(\mu\right),\left\Vert \cdot\right\Vert _{*}^{\prime}\right)$.
In particular, $T$ is an isometry of $\left(L^{2}\left(\mu\right)\cap L^{\infty}\left(\mu\right),\left\Vert \cdot\right\Vert _{*}^{\prime}\right)$.
\end{prop}

\begin{proof}
The first step consists in showing that $T_{*}$ preserves $\ker\pi$:

Take $Z\in\ker\pi$, in particular $Z$ belongs to $\mathcal{I}_{2}^{\perp}\subset L^{2}\left(\mu^{*}\right)$.
But $T_{*}$ commutes with $P$ thanks to Lemma \ref{lem:projection}
and thus $P\left(Z\circ T_{*}\right)=P\left(Z\right)\circ T_{*}=0$
and thus $Z\circ T_{*}\in\ker\pi$.

Now observe that for any $Z\in L^{\infty}\left(\mu^{*}\right)$:
\begin{align*}
\pi\left(Z\circ T_{*}\right) & =I_{1}^{-1}\left(P\left(Z\circ T_{*}\right)\right)\\
 & =I_{1}^{-1}\left(P\left(Z\right)\circ T_{*}\right)\\
 & =I_{1}^{-1}\left(P\left(Z\right)\right)\circ T\\
 & =\pi\left(Z\right)\circ T
\end{align*}

and this proves the result.
\end{proof}
\begin{defn}
\label{def:DualT_P}We define the \emph{Poisson-transfer operator}
$\widehat{T}_{\mathcal{P}}$ acting on $L^{\Phi}\left(\mu\right)$
as the dual operator of $T$ acting on $L^{\Phi}\left(\mu\right)^{\prime}$.
It satisfies, for any $f$ in $L^{\Phi}\left(\mu\right)$, $g\in L^{2}\left(\mu\right)\cap L^{\infty}\left(\mu\right)$
:
\begin{align*}
\int_{X}\widehat{T}_{\mathcal{P}}\left(f\right)\cdot gd\mu & =\int_{X}fg\circ Td\mu.
\end{align*}
\end{defn}

\begin{prop}
\label{prop:CoincidenceDual}$\widehat{T_{*}}$ preserves $\mathcal{I}_{1}$
and satisfies, for all $f\in L^{\Phi}\left(\mu\right)$:
\[
\widehat{T_{*}}I_{1}\left(f\right)=I_{1}\left(\widehat{T}_{\mathcal{P}}f\right)
\]

Moreover $\widehat{T}_{\mathcal{P}}$ preserves $L^{1}\left(\mu\right)$
and coincide with $\widehat{T}$ on it.
\end{prop}

\begin{proof}
The first part of the proof follows immediately from Proposition \ref{prop:Tisomdual}
and the fact that $\ker\pi=\mathcal{I}_{1}^{\perp}$ from Proposition
\ref{prop:dualitystructure}.

Now taking $f\in L^{1}\left(\mu\right)\subset L^{\Phi}\left(\mu\right)$,
for any set $A\in\mathcal{A}_{f}$:
\[
\int_{A}\widehat{T}_{\mathcal{P}}\left(f\right)d\mu=\int_{T^{-1}A}fd\mu=\int_{A}\widehat{T}\left(f\right)d\mu.
\]

and this implies $\widehat{T}_{\mathcal{P}}f=\widehat{T}f$, that
is $\widehat{T}_{\mathcal{P}}$ coincide with $\widehat{T}$ on $L^{1}\left(\mu\right)$.
\end{proof}
\begin{rem}
Being a positive operator defined on $L^{1}\left(\mu\right)$, $\widehat{T}$
extends in a unique (mod. $\mu$) way to any non-negative measurable
function. The same applies to $\widehat{T}_{\mathcal{P}}$ and since
they coincide on $L^{1}\left(\mu\right)$, the extension of $\widehat{T}$
to $L^{\Phi}\left(\mu\right)$ is in fact $\widehat{T}_{\mathcal{P}}$.
\end{rem}

\section{\label{sec:exact-and-remotely}exact and remotely infinite transformations}

We recall the two definitions that we will deal with in this section:
\begin{defn}
(See \cite{KrenSuch69MixInf}) $\left(X,\mathcal{A},\mu,T\right)$
is said to be 
\begin{itemize}
\item \emph{exact} if $\cap_{n\ge0}T^{-n}\mathcal{A}=\left\{ \emptyset,X\right\} $
mod. $\mu$.
\item \emph{remotely infinit}e\footnote{We warn the reader that the authors of \cite{KrenSuch69MixInf} also
define, \emph{remotely infinite automorphisms} (which cannot exist
with the above definition...) in the same paper. This is the same
difference that occurs between exact endomorphisms and $K$-automorphisms.} if $\cap_{n\ge0}T^{-n}\mathcal{A}$ contains only $0$ or infinite
measure sets.
\end{itemize}
\end{defn}

We recall the chain of implications we already gave in the introduction:

\[
T\:\text{exact}\:\Rightarrow T\:\text{remotely infinite}\:\Rightarrow T\:\text{of zero type}\Longrightarrow\text{absence of a.c. \ensuremath{T}-invariant measure}.
\]

Note also that exactness implies ergodicity whereas it is not the
case for the other three properties.
\begin{example}
It is be proven (see Corollary 3.7.7 in \cite{Aar97InfErg}) that
pointwise dual ergodic endomorphisms (see also \cite{Aar97InfErg}
for the definition and examples) are always remotely infinite. Null
recurrent aperiodic Markov chains under their infinite invariant measure
yields unilateral shifts that are exact endomorphisms.
\end{example}

We are interested in transfer operator characterization of these properties,
we also recall Lin's result:
\begin{thm}
\label{thm:lin}(\cite{Lin1971}) $\left(X,\mathcal{A},\mu,T\right)$
is exact if and only if, for all $f\in L_{0}^{1}\left(\mu\right)$:
\[
\left\Vert \widehat{T^{n}}f\right\Vert _{L^{1}\left(\mu\right)}\to0,
\]

as $n\to+\infty$.
\end{thm}

We will prove the following:
\begin{thm}
\label{thm:remotelyDualOp}$\left(X,\mathcal{A},\mu,T\right)$ is
remotely infinite if and only if, for all $f\in L_{0}^{1}\left(\mu\right)$:
\[
\left\Vert \widehat{T^{n}}f\right\Vert _{*}\to0,
\]

as $n\to+\infty$.
\end{thm}

To this end, we first need to establish the following characterization
of remotely infinite systems:
\begin{prop}
\label{prop:exactPoissonRemoteBase}$\left(X,\mathcal{A},\mu,T\right)$
is remotely infinite if and only if $\left(X^{*},\mathcal{A}^{*},\mu^{*},T_{*}\right)$
is exact.
\end{prop}

\begin{proof}
Recall (see \cite{Roy07Infinite}) that if $\mathcal{B}\subset\mathcal{A}$
is a $\sigma$-algebra then $\mathcal{B}^{*}:=\sigma\left\{ N\left(A\right),\:A\in\mathcal{B}\right\} $
is said to be a \emph{Poisson $\sigma$-algebra}.

The proof of the Proposition is then a straightforward consequence
of the following facts
\end{proof}
\begin{itemize}
\item $\left(T^{-n}\mathcal{A}\right)^{*}=T_{*}^{-n}\mathcal{A}^{*}$
\item $\left(\cap_{n\ge0}T^{-n}\mathcal{A}\right)^{*}=\cap_{n\ge0}T_{*}^{-n}\mathcal{A}^{*}$
(see Lemma 3.3 in \cite{Roy07Infinite} for a proof)
\item A Poisson $\sigma$-algebra $\mathcal{G}^{*}$ for $\mathcal{G}\subset\mathcal{A}$
is trivial if and only if $\mathcal{G}$ contains only zero or infinite
measure sets (the Poisson random measure $N$ cannot distinguish between
zero measure sets nor between infinite measure sets).
\end{itemize}
We can then prove our theorem:
\begin{proof}
\emph{(of Theorem \ref{thm:remotelyDualOp})}. Assume $\left(X,\mathcal{A},\mu,T\right)$
is remotely infinite. Then $\left(X^{*},\mathcal{A}^{*},\mu^{*},T_{*}\right)$
is exact from Proposition \ref{prop:exactPoissonRemoteBase}, and
thus, for any $f\in L_{0}^{1}\left(\mu\right)$,
\[
\left\Vert \widehat{T_{*}^{n}}I_{1}\left(f\right)\right\Vert _{L^{1}\left(\mu^{*}\right)}\to0
\]

But, we have seen in Proposition \ref{prop:CoincidenceDual} that
$\widehat{T_{*}^{n}}I_{1}\left(f\right)=I_{1}\left(\widehat{T^{n}}f\right)$,
therefore
\begin{align*}
\left\Vert \widehat{T_{*}^{n}}I_{1}\left(f\right)\right\Vert _{L^{1}\left(\mu^{*}\right)} & =\left\Vert I_{1}\left(\widehat{T^{n}}f\right)\right\Vert _{L^{1}\left(\mu^{*}\right)}\\
 & =\left\Vert \widehat{T^{n}}f\right\Vert _{*}.
\end{align*}

For the converse, assume that for all $f\in L_{0}^{1}\left(\mu\right)$,
$\left\Vert \widehat{T^{n}}f\right\Vert _{*}\to0$ as $n\to+\infty$.

Recall that $\widehat{T^{n}}_{\mathcal{P}}$ is an isometry on $L^{\Phi}\left(\mu\right)$
and from Proposition \ref{prop:CoincidenceDual}, $\widehat{T^{n}}=\widehat{T^{n}}_{\mathcal{P}}$
on $L_{0}^{1}\left(\mu\right)$, thus, for any $g\in L^{\Phi}\left(\mu\right)$
and $f\in L_{0}^{1}\left(\mu\right)$:

\begin{align*}
\left\Vert \widehat{T^{n}}_{\mathcal{P}}g\right\Vert _{*} & =\left\Vert \widehat{T^{n}}_{\mathcal{P}}\left(g-f\right)+\widehat{T^{n}}_{\mathcal{P}}f\right\Vert _{*}\\
 & \le\left\Vert \widehat{T^{n}}_{\mathcal{P}}\left(g-f\right)\right\Vert _{*}+\left\Vert \widehat{T^{n}}_{\mathcal{P}}f\right\Vert _{*}\\
 & \le\left\Vert g-f\right\Vert _{*}+\left\Vert \widehat{T^{n}}f\right\Vert _{*}
\end{align*}

That is
\[
\limsup_{n\to\infty}\left\Vert \widehat{T^{n}}_{\mathcal{P}}g\right\Vert _{*}\le\left\Vert g-f\right\Vert _{*}.
\]

We conclude by density of $L_{0}^{1}\left(\mu\right)$ in $L^{\Phi}\left(\mu\right)$
(see Theorem \ref{thm:Orlicz}).

In other words, from Proposition \ref{prop:CoincidenceDual} we have
proved that, for any $F\in\mathcal{I}_{1}$:

\[
\left\Vert \widehat{T_{*}^{n}}F\right\Vert _{L^{1}\left(\mu^{*}\right)}\to0
\]

as $n\to+\infty$.

For the converse we follow the same lines of arguments used in the
proof of Theorem \ref{thm:lin}: assume that $\left(X,\mathcal{A},\mu,T\right)$
is not remotely infinite, this means that there exists $A\in\cap_{n\ge0}T^{-n}\mathcal{A}$
such that $0<\mu\left(A\right)<\infty$ and we can set, for any $n\in\mathbb{N}$,
$A_{n}$ such that $A=T^{-n}A_{n}$.

Consider $1_{N\left(A\right)=0}\in L^{\infty}\left(\mu^{*}\right)$,
we have $1_{N\left(A_{n}\right)=0}\circ T_{*}^{n}=1_{N\left(T^{-n}A_{n}\right)=0}=1_{N\left(A\right)=0}$.

Then, for any $n\in\mathbb{N}$,
\begin{align*}
\left\Vert \widehat{T_{*}^{n}}I_{1}\left(-1_{A}\right)\right\Vert _{L^{1}\left(\mu^{*}\right)} & \ge\mathbb{E}_{\mu^{*}}\left[\left|1_{N\left(A_{n}\right)=0}\cdot\widehat{T_{*}^{n}}I_{1}\left(-1_{A}\right)\right|\right]\\
 & \ge\mathbb{E}_{\mu^{*}}\left[1_{N\left(A_{n}\right)=0}\cdot\widehat{T_{*}^{n}}I_{1}\left(-1_{A}\right)\right]\\
 & \ge\mathbb{E}_{\mu^{*}}\left[1_{N\left(A_{n}\right)=0}\circ T_{*}^{n}\cdot I_{1}\left(-1_{A}\right)\right]\\
 & =\mathbb{E}_{\mu^{*}}\left[1_{N\left(A\right)=0}\cdot I_{1}\left(-1_{A}\right)\right]\\
 & =\mathbb{E}_{\mu^{*}}\left[1_{N\left(A\right)=0}\cdot\left(-N\left(A\right)+\mu\left(A\right)\right)\right]\\
 & =\mu\left(A\right)\mathbb{E}_{\mu^{*}}\left[1_{N\left(A\right)=0}\right]\\
 & =\mu\left(A\right)e^{-\mu\left(A\right)}>0.
\end{align*}

This contradicts that for any $F\in\mathcal{I}_{1}$, $\left\Vert \widehat{T_{*}^{n}}F\right\Vert _{L^{1}\left(\mu^{*}\right)}\to0$
.
\end{proof}
Combining Proposition \ref{thm:remotelyDualOp} and the last point
of Theorem \ref{thm:Orlicz}, we obtain:
\begin{cor}
$\left(X,\mathcal{A},\mu,T\right)$ is not remotely infinite if there
exists a non-negative $f\in L^{1}\left(\mu\right)$ such that
\[
\liminf_{n\to\infty}\widehat{T^{n}}f\neq0\;\mu\text{-a.e.}
\]
\end{cor}

Observe that Fatou lemma gives, if $\liminf_{n\to\infty}\widehat{T^{n}}f\neq0$:
\[
0<\int_{X}\left(\liminf_{n\to\infty}\widehat{T^{n}}f\right)d\mu\le\liminf_{n\to\infty}\left\Vert \widehat{T^{n}}f\right\Vert _{1}
\]

which only implies that $\left(X,\mathcal{A},\mu,T\right)$ is not
exact.

\section{Conclusion}

We aimed to have illustrated the canonical character of $L^{\Phi}\left(\mu\right)$
and think that this space could be a new reference space when studying
infinite measure dynamical systems. If the Poisson-Orlicz norm is
by far the preferred choice when looking at the interplay between
a system and its Poisson suspension as we did, the actual computation
of the norm of a particular function is a mathematical challenge on
its own ! However the equivalent Orlicz norm is always available and
much easier to compute in concrete situations, making the use of $L^{\Phi}\left(\mu\right)$
quite flexible.

\bibliographystyle{plain}
\bibliography{biblioNew.bib}

\end{document}